\documentclass{gtpart}

\usepackage{amssymb}

\usepackage{graphicx}

\newtheorem{theorem}{Theorem}[section]

\newtheorem{proposition}[theorem]{Proposition}
\newtheorem{corollary}[theorem]{Corollary}

\theoremstyle{definition}
\newtheorem{definition}[theorem]{Definition}

\theoremstyle{remark}
\newtheorem{remark}[theorem]{Remark}

\newcommand{\TryPackage}[3]{\IfFileExists{#1.sty}{\usepackage{#1}#2}{#3}}
\TryPackage{mathrsfs}{\renewcommand{\mathcal}{\mathscr}}{%
     \TryPackage{eucal}{}{}}

\newcommand{\lto}{\longrightarrow}
\newcommand{\al}{\alpha}
\newcommand{\be}{\beta}

\newcommand{\ga}{\gamma}
\newcommand{\ep}{\varepsilon}
\newcommand{\varep}{\varepsilon}

\newcommand{\la}{\lambda}

\newcommand{\si}{\sigma}
\newcommand{\Ga}{\Gamma}

\newcommand{\La}{\Lambda}
\newcommand{\Si}{\Sigma}

\newcommand{\ZZ}{{\mathbb Z}}
\newcommand{\RR}{{\mathbb R}}
\newcommand{\CC}{{\mathbb C}}

\newcommand{\PP}{{\mathbb P}}
\newcommand{\QQ}{{\mathbb Q}}

\newcommand{\cE}{{\mathcal E}}

\newcommand{\mer}{{\mathcal M}} 
\newcommand{\lng}{{\mathcal L}}  
\newcommand{\SLC}{{SL(2, {\mathbb C})}}

\newcommand{\tr}{\operatorname{\it tr}}

\newcommand{\bn}{{\textit{\textbf{n}}}}

\newcommand{\sm}{{\smallsetminus}}

\begin{document}

\title[The SL$(2,\CC)$ Casson invariant for Dehn surgeries on two-bridge knots]
{The SL$(2,\CC)$ Casson invariant for Dehn \\ surgeries on two-bridge knots}


\author{Hans U. Boden}
\address{Mathematics \& Statistics, McMaster University, Hamilton, Ontario, L8S 4K1 Canada}
\email{boden@mcmaster.ca}
\thanks{The first author was supported by a grant from the Natural Sciences and Engineering Research Council of Canada.}

\author{Cynthia L. Curtis}
\address{Mathematics \& Statistics, The College of New Jersey, Ewing, NJ, 08628 USA}
\email{ccurtis@tcnj.edu}
\thanks{}

\subject{primary}{MSC2010}{57M27}
\subject{secondary}{MSC2010}{57M25}
\subject{secondary}{MSC2010}{57M05}


\dedicatory{}


\date{\today}
\begin{abstract}
We investigate the behavior of the $\SLC$ Casson invariant
for 3-manifolds  obtained by Dehn surgery along two-bridge knots.
Using the results of Hatcher and Thurston, and also results of Ohtsuki, 
we outline how to compute the Culler--Shalen seminorms, and
we illustrate this approach by providing explicit
computations  for double twist knots. We then apply the surgery formula 
to deduce the $\SLC$ Casson invariant for the 3-manifolds obtained by $(p/q)$--Dehn surgery on such knots. These results are applied to prove nontriviality of the $\SLC$ Casson invariant for nearly all  3-manifolds obtained by nontrivial Dehn surgery on a hyperbolic two-bridge knot.
We relate the formulas derived to degrees of $A$-polynomials and use this information to identify factors of higher multiplicity in the $\widehat{A}$-polynomial, which is the $A$-polynomial with multiplicities as defined by Boyer-Zhang.
\end{abstract}

\maketitle


\section*{Introduction}
The goal of this paper is to provide computations of the $\SLC$ Casson invariant for 3-manifolds obtained
by Dehn surgery on a two-bridge knot.
Our approach is to apply the Dehn surgery formula of
\cite{C} and \cite{C1}, and this involves computing the Culler--Shalen seminorms.
In general, the surgery formula applies to Dehn surgeries on small knots $K$ in homology spheres $\Si$,
and a well-known result of Hatcher and Thurston \cite{HT} shows that
 all two-bridge knots are small. The Culler--Shalen seminorm plays a key role in the surgery formula, and we use the results of Ohtsuki \cite{O} to provide the required computations.

As an application, using the classification of exceptional Dehn surgeries on two-bridge knots from \cite{BW}, we prove that nearly all 3-manifolds
given by a nontrivial $(p/q)$--Dehn surgery on a hyperbolic two-bridge knot $K$ have nontrivial $\SLC$ Casson invariant.
 In our previous work \cite[Theorem 2.2]{BC2}, we showed that for any two-bridge knot or torus knot $K$, the homology 3-sphere
obtained by $(1/q)$--Dehn surgery on $K$ has nonzero $\SLC$ Casson invariant whenever $q \neq 0$, and Theorem \ref{nontriv} gives a new proof extending this result to nearly all $p/q$ surgeries.

Our second application is to $A$-polynomials. The $A$-polynomial as originally defined in \cite{CCGLS} has no repeated factors, but in \cite{BZ2}, Boyer and Zhang give an alternative approach by
 using the degree of the restriction map to assign multiplicities to each one-dimensional component in the character variety.
For a given knot, we denote the original $A$-polynomial by $A_K(M,L)$ and the Boyer-Zhang polynomial by $\widehat{A}_K(M,L)$
 (see Section \ref{A-polynomial}).
If $K$ is a small knot, then  $\widehat{A}_K(M,L)$ and $A_K(M,L)$ have the same irreducible factors, only $\widehat{A}_K(M,L)$ may include factors of higher multiplicity.  We exploit the close relationship between the $\SLC$ Casson invariant, Culler--Shalen seminorms, and the $\widehat{A}$-polynomial to determine closed formulas for the $L$-degrees of $\widehat{A}_K(M,L)$ for all two-bridge knots and for the $M$-degrees of $\widehat{A}_K(M,L)$ for double twist knots. Our techniques also enable computations of the $M$-degree of $\widehat{A}_K(M,L)$ for other two-bridge knots on a case-by-case basis. By comparing our results to known calculations of the $A$-polynomial, we are able to identify knots for which $A_K(M,L) \neq \widehat{A}_K(M,L)$.  In some cases, we are also able to determine the multiplicities of the factors of $\widehat{A}_K(M,L),$ and we illustrate this for the examples of the knots $7_4$ and $8_{11}$.

We briefly outline the contents of this paper. We begin with an introduction of the notation for
the $\SLC$ character varieties, the definition of the $\SLC$ Casson invariant, Culler--Shalen seminorms, and the surgery formula for the $\SLC$ Casson invariant. We then present the two-bridge knots and establish a regularity result for all slopes. We also determine the multiplicities of the curves in the character variety; in fact we show they all equal one. Following \cite{O}, we explain how to calculate the Culler--Shalen seminorms,  and we provide explicit calculations for the family of examples given by the double twist knots $K=J(\ell,m)$ (see Figure \ref{gentwist}). We use this to compute the $\SLC$ Casson invariant
for 3-manifolds obtained by Dehn surgery on $K$. We prove a nontriviality result for the 3-manifold invariant for most Dehn surgeries on hyperbolic two-bridge knots. In the final section we introduce the $A$-polynomial and $\widehat{A}$-polynomial, and we apply our results to make some general comments about the $L$- and $M$-degrees of these polynomials for two-bridge knots,
and we show that the corresponding $\SLC$ knot invariant is nontrivial for all small knots. In the appendix, we list the information on boundary slopes needed to calculate the $\SLC$ Casson invariant for surgeries on 2-bridge knots with up to 8 crossings, and we also list there the $L$- and $M$-degrees of $A_K(M,L)$ and $\widehat{A}_K(M,L)$.

\section{Preliminaries}
Given a finitely generated group $\Ga$, we set $R(\Ga)$ to be the
space of representations $\varrho\colon \Ga \lto \SLC$ and
$R^*(\Ga)$ the subspace of irreducible representations. Recall
from \cite{CS1} that $R(\Ga)$ has the structure of a complex affine
algebraic set. The {\sl character}  of a representation $\varrho$ is
the function $\chi_\varrho\colon \Ga \lto \CC$ defined by setting
$\chi_\varrho(g)=\tr(\varrho(g))$ for $\ga \in \Ga$. The set
of characters of $\SLC$ representations is denoted $X(\Ga)$ and
also admits the structure of a complex affine algebraic set.
Let  $X^*(\Ga)$ denote the subspace of characters of
irreducible representations.
Define $t\colon R(\Ga) \lto
X(\Ga)$ by $\varrho \mapsto \chi_\varrho$, and note that $t$ is
surjective.  Given a manifold $\Si$, we denote by
$R(\Si)$ the variety of $\SLC$ representations of $\pi_1 \Si$ and
by $X(\Si)$ the associated character variety.

We briefly recall the definition of the $\SLC$ Casson invariant.
Suppose now $\Si$ is a closed, orientable 3-manifold
with a Heegaard splitting $(W_1, W_2, F)$. Here, $F$
is a closed orientable surface embedded in $\Si$, and $W_1$ and $W_2$ are handlebodies
with boundaries $\partial W_1 = F = \partial W_2$
 such that $\Si = W_1\cup_F W_2$.
The inclusion maps $F \hookrightarrow W_i$ and $W_i \hookrightarrow \Si$
induce surjections of fundamental groups. On the level of character varieties, this identifies $X(\Si)$  as the intersection
$$X(\Si) = X(W_1) \cap X(W_2) \subset X(F).$$

There are natural orientations on all the character varieties
determined by their complex structures. The  invariant
$\la_\SLC(\Si)$ is defined as an oriented intersection number of
$X^*(W_1)$ and $X^*(W_2)$ in $X^*(F)$ which counts only compact,
zero-dimensional components of the intersection. Specifically,
there exist a compact neighborhood $U$ of the zero-dimensional
components of $X^*(W_1)\cap X^*(W_2)$ which is disjoint from the
higher dimensional components of the intersection and an isotopy
$h\colon X^*(F) \to X^*(F)$ supported in $U$ such that
$h(X^*(W_1))$ and $X^*(W_2)$ intersect transversely in $U$.
Given a zero-dimensional component $\{\chi\}$ of
$h(X^*(W_1))\cap X^*(W_2)$, we set $\varep_\chi = \pm 1$,
depending on whether the orientation of $h(X^*(W_1))$  followed by
that of $X^*(W_2)$ agrees with or disagrees with the orientation
of $X^*(F)$ at $\chi$.

\begin{definition} Let
$\la_\SLC(\Si) = \sum_\chi \varep_\chi,$
where the sum is  over all zero-dimensional
components  of the intersection $h(X^*(W_1))\cap X^*(W_2)$.
\end{definition}

In \cite{C}, a surgery formula was established for $\la_\SLC(\Si)$ for Dehn surgeries on small knots in integral homology 3-spheres. This result is of central importance in this paper, so we introduce the background material on Culler--Shalen seminorms and recall the statement of the theorem.

Now suppose $M$ is a compact, irreducible, orientable 3-manifold with
boundary a torus. An {\sl incompressible surface} in $M$ is a
properly embedded surface $(F,\partial F) \hookrightarrow (M,\partial M)$ such
that $\pi_1 F \lto \pi_1 M$ is injective and no component of $F$ is
a 2-sphere bounding a 3-ball. A surface $F$ in $M$ is called {\sl essential} if
it is incompressible and has no boundary parallel components.
The manifold $M$ is called {\sl small} if it
does not contain a closed essential surface, and a knot $K$ in
$\Si$ is called {\sl small} if its complement $\Si \sm
\tau(K)$ is a small manifold, where $\tau(K)$ is a tubular neighborhood of $K$.

If $\ga$ is a simple closed curve in $\partial M$, the Dehn
filling of $M$ along $\ga$ will be denoted by $M(\ga)$; it is the
closed 3-manifold obtained by identifying a solid torus with $M$
along their boundaries so that $\ga$ bounds a disk. Note that the
homeomorphism type of $M(\ga)$ depends only on the {\sl slope} of
$\ga$ -- that is, the unoriented isotopy class of $\ga$. Primitive
elements in $H_1(\partial M; \ZZ)$ determine slopes
under a two-to-one correspondence.

If $F$ is an essential surface in $M$
with nonempty boundary, then all of its  boundary
components are parallel and the slope of one (hence all) of these
curves is called the {\sl boundary slope} of $F$. A
slope is called a {\sl strict boundary slope} if it is
the boundary slope of an essential surface that is not the
fiber of any fibration of $M$ over $S^1$.

For each $\ga \in \pi_1 M,$ there is a regular map $I_\ga\colon X(M)
\lto \CC$ defined by $I_\ga(\chi) = \chi(\ga).$ Let
$e\colon H_1(\partial M;\ZZ)\lto \pi_1(\partial M)$ be the
inverse of the Hurewicz
isomorphism.  Identifying $e(\xi) \in \pi_1(\partial M)$ with its
image in $\pi_1 M$ under the natural map $\pi_1(\partial M) \lto
\pi_1 M,$ we obtain a well-defined function $I_{e(\xi)}$ on $X(M)$
for each $\xi \in H_1(\partial M; \ZZ).$ Let $f_\xi\colon X(M) \lto \CC$
be the regular function defined by $f_\xi  = I_{e(\xi)} -2$ for
$\xi \in H_1(\partial M; \ZZ).$

Let $r\colon X(M)\lto X(\partial M)$ be the restriction map induced by
$\pi_1(\partial M) \lto \pi_1 M$. Suppose $X_i$ is an algebraic
component of $X(M)$ with $\dim X_i =1$ such that $r(X_i)$ is
also one-dimensional.  Let $f_{i,\xi}\colon X_i \lto \CC$ denote the regular
function obtained by restricting $f_{\xi}$ to $X_i$ for each $i$.

Let $\widetilde{X}_i$ denote the smooth, projective curve
birationally equivalent to $X_i$. Regular functions on $X_i$
extend to rational functions on $\widetilde{X}_i$. We abuse
notation and denote the extension of $f_{i,\xi}$ to
$\widetilde{X}_i$ also by $f_{i, \xi}\colon \widetilde{X}_i \lto \CC
\cup \{ \infty \} = \CC \PP^1.$

If $M$ is hyperbolic, then Proposition 3.1.1 of \cite{CS1} establishes the existence of
a discrete faithful irreducible representation $\varrho_0 : \pi_1(M) \to \SLC,$
and Proposition 2 of \cite{CS2} implies that
the algebraic component $X_0$ of $X(M)$ containing the character $\chi_{\varrho_0}$ is one-dimensional.
In Section 1.4 of \cite{CGLS}, the authors construct a norm $\| \cdot \|_{0}$
on the real vector space $H_1(\partial
M; \RR)$ called the Culler--Shalen norm
associated to $X_0$. The same construction works
 for other dimension one components $X_i$ of
$X(M)$, and in general, one obtains a seminorm $\| \cdot \|_{i}$  on  $H_1(\partial
M; \RR)$ called the Culler--Shalen seminorm associated to $X_i$.

 \begin{definition} \label{CS-seminorm}
For $X_i$ a one-dimensional component of
$X(M)$ containing an irreducible character and whose
restriction $r(X_i)$ is also one-dimensional, define
the seminorm  $\| \cdot \|_{i}$ on $H_1(\partial M; \RR)$   by setting
$$ \| \xi \|_{i} =  \deg(f_{i,\xi})$$
for all $\xi$ in the lattice  $H_1(\partial M; \ZZ)$.
We call a one-dimensional algebraic component $X_i$ of $X(M)$   a {\sl norm curve}
if $\| \cdot \|_i$ defines a norm on $H_1(\partial M; \RR)$.
\end{definition}

The $\SLC$ Casson invariant of a manifold obtained
by Dehn filling is closely related to this seminorm; however we must impose certain
restrictions on the  slope of the Dehn filling.

\begin{definition} \label{irregular}
The slope of a simple closed curve $\ga$ in $\partial M$
is called {\sl irregular} if there exists an irreducible representation
$\varrho\colon \pi_1 M \lto \SLC$ such that
\begin{itemize}
 \item [(i)] the character $\chi_{\varrho}$ of $\varrho$ lies on a one-dimensional
component $X_i$ of $X(M)$ such that $r(X_i)$ is one-dimensional,
 \item [(ii)] $\tr \varrho (\al) =\pm 2$ for all
 $\al$ in the image of $i^*\colon \pi_1(\partial M) \lto \pi_1(M),$
 \item[(iii)] $\ker (\varrho \circ i^*)$ is the cyclic group generated by $[\ga] \in \pi_1(\partial M)$.
\end{itemize}
A slope is called {\sl regular} if it is not irregular.
\end{definition}

With these definitions, we are almost ready to state the
Dehn surgery formula. But first we recall some useful notation for Dehn fillings of knot complements.
For any choice of basis $(u,v)$  for $H_1(\partial M; \ZZ),$ there is
a bijective correspondence between unoriented isotopy classes of
simple closed curves in $\partial M$ and elements in
$\QQ \cup \{ \infty\}$ given by $\ga\mapsto p/q,$ where  $\ga = p u + q v$.
If $M$ is the complement of a knot $K$ in an integral homology
3-sphere $\Si$, then the meridian $\mer$ and preferred longitude $\lng$ of the knot $K$
provide a basis for $H_1(\partial M; \ZZ).$ Consider the 3-manifold $M(\ga)$ resulting from  Dehn filling along the curve $\ga = p \mer + q \lng.$ In this case, we call $p/q$ the slope of $\ga$ and denote by $M_{p/q} = M(\ga)$ the 3-manifold obtained by  $(p/q)$--Dehn surgery along the knot $K$.

\begin{definition}
A slope $p/q$ is called {\sl admissible} for $K$ if
\begin{enumerate}
\item[(i)] $p/q$ is a regular slope which is not a strict boundary slope, and
\item[(ii)] no $p'$-th root of unity is a root of the Alexander polynomial of
$K$, where $p'=p$ if $p$ is odd and $p' = p/2$ if $p$ is even.
\end{enumerate}
\end{definition}

The next result is a restatement of Theorem 4.8 of \cite{C}, as
corrected in \cite{C1}.
\begin{theorem} \label{surg-form}
Suppose $K$ is a small knot in an integral homology 3-sphere $\Si$ with
complement $M$. Let $\{X_i\}$ be the collection of all
one-dimensional components of the character variety $X(M)$ such
that $r(X_i)$ is one-dimensional and such that $X_i \cap X^*(M)$
is nonempty. Define $\si\colon \ZZ \lto \{0,1\}$ by $\si(p) =0$ if
$p$ is even and $\si(p)=1$ if $p$ is odd.

Then there exist integral weights $m_i > 0$ depending only on
$X_i$ and non-negative numbers $E_0, E_1 \in \frac{1}{2}
\ZZ$ depending only on $K$ such that for every admissible slope
$p/q$, we have
$$
\la_\SLC (M_{p/q}) = \frac{1}{2} \sum_i m_i \| p\mer + q \lng \|_{i} - E_{\si(p)}.
$$
\end{theorem}
We define the total Culler--Shalen seminorm $\| p/q \|_{T} =  \sum_i m_i \| p\mer + q \lng \|_{i}.$ We note that this is half the norm defined by Culler-Gordon-Luecke-Shalen in \cite{CGLS} and extended to the seminorm case in \cite{BZ1} and \cite{BZ2}. It is twice the norm defined in \cite{BC1}.

We
 briefly recall some useful properties of the $\SLC$ Casson invariant and
 we refer to   \cite{C} and \cite{BC1} for further details.

On closed 3-manifolds $\Si,$ the invariant $\la_\SLC(\Si) \geq 0$ is nonnegative, satisfies $\la_\SLC(-\Si) = \la_\SLC(\Si)$ under
orientation reversal, and is additive  under connected sum  of $\ZZ_2$--homology  3-spheres
  (cf.~Theorem 3.1, \cite{BC1}). If $\Si$ is hyperbolic, then $\la(\Si) > 0$
by Proposition 3.2 of \cite{C}.

If $K$ is a small knot
in an integral homology 3-sphere $\Si$, then
Theorem \ref{surg-form} implies that  the difference
$\la_\SLC (M_{p/(q+1)}) - \la_\SLC (M_{p/q})$
is independent of $p$ and $q$ provided $q$ is chosen sufficiently large.
This allows one to define an  invariant of small knots $K$ in homology 3-spheres by setting
\begin{equation}\label{knotinv}
\la'_\SLC(K) = \la_\SLC (M_{p/(q+1)}) - \la_\SLC (M_{p/q})
\end{equation}
for $q$   sufficiently large.

\section{Two-bridge knots}
For the remainder of this paper, $K=K(\al,\be)$ will denote a two-bridge knot in $S^3$
with complement $M = S^3 \sm \tau(K)$.
Recall from \cite[Chapter 12]{BuZi} that given relatively prime integers $\al,\be$ with $\al>1$ odd, we can associate a knot with two bridges, denoted $K(\al,\be),$ whose associated
knot group $G(\al,\be)= \pi_1(S^3 \sm \tau(K))$
admits the presentation
\begin{equation} \label{knotgroup}
G(\al,\be)=\langle x,y \mid
xw=wy\rangle,
\end{equation}
where $w=y^{\ep_1}x^{\ep_2} \cdots y^{\ep_{\al-2}} x^{\ep_{\al-1}}$
for $\ep_i = (-1)^{\lfloor {i \be}/{\al} \rfloor}$.
 Two such knots $K(\al,\be)$ and $K(\al',\be')$ are equivalent
 if and only if
$\al' = \al$ and $\be' = \be^{\pm 1} \mod \al.$ Thus we can choose $\be$ so that $0<\be < \al$.

In this section, we will outline how to apply Theorem \ref{surg-form} to compute the $\SLC$ Casson invariant for 3-manifolds obtained by $(p/q)$--Dehn surgery on such a knot. Recall that by the results of Hatcher and Thurston \cite{HT}, it follows that all two-bridge knots are small. Further, we know that the $\SLC$ Casson invariant  does not change under orientation reversal (cf.  Theorem 1.2, \cite{BC1}), and since the manifolds obtained by $(p/q)$--Dehn surgery on a knot $K$ are, up to change of orientation, the same as those obtained by $(-p/q)$--Dehn surgery on its mirror image $\overline{K}$, for our purposes, we can choose to work with $K$ or $\overline{K}.$ Note further that the mirror image of the two-bridge knot $K=K(\al,\be)$ is just $\overline{K} = K(\al,-\be) =K(\al,\al-\be)$ (cf. p.185, \cite{BuZi}).

\subsection{Regularity of surgery slopes}
In this section, we show that every slope $p/q$ on a two-bridge knot is regular.
This amounts to showing that any irreducible
$\SLC$ representation $\varrho$ of a two-bridge knot group
with $\chi_\varrho(\mer)=\pm 2$ and $\chi_\varrho(\lng)= \pm 2$
has $\varrho(\lng)^q \neq \varrho(\mer)^p.$

\begin{proposition} \label{reg}
If $K$ is a $2$-bridge knot, then every
slope $p/q$ is regular.
\end{proposition}

\begin{proof}
Choosing $\be$ odd with $-\al< \be < \al$ and writing the knot group $G(\al,\be)$ as in Equation \eqref{knotgroup},
we note that $\ep_{\al-i}=\ep_i$ for $i=1,\ldots, \al-1,$
hence the integer
$$n := -\sum_{i=1}^{\al-1} \ep_i = -2\sum_{i=1}^{(\al-1)/2} \ep_i$$
is even.

Let $w^*$ be the word in $x,y$ obtained by reversing $w$, specifically
$$w^* = x^{\ep_1}y^{\ep_2} \cdots x^{\ep_{\al-2}} y^{\ep_{\al-1}}.$$
Then the meridian and longitude are
$\mer=x$ and $\lng=x^{2n}w w^*$.

As in
section 7 of \cite{CCGLS} and in section 1 of \cite{Riley},  up to conjugation, we can assume that any irreducible representation $\varrho\colon
G(\al,\be)\lto \SLC$ sends
$$\varrho(x) = \left[ \begin{array}{cc}
\mu & 1\\
0 & \mu^{-1} \end{array} \right] \quad \text{and} \quad \varrho(y) =
\left[ \begin{array}{cc}
\mu & 0\\
t & \mu^{-1} \end{array} \right], $$ where $t$ is
chosen so that $\varrho(xw) = \varrho(wy)$.

Suppose $p/q$ is an irregular slope, where $p,q$ are taken to be relatively prime.
Then there exists an
irreducible representation $\varrho$ of the knot complement satisfying the conditions of Definition \ref{irregular}, so in particular
$\chi_{\varrho}(\mer) = \pm 2, \chi_{\varrho}(\lng) = \pm 2,$ and $
\varrho(\mer^p \lng^q) = I.$  We will show that this leads to a contradiction.

Since $\chi_\varrho(\mer) =\pm 2$, we know that $\mu = \mu^{-1}= \pm 1$.
Each of the four possible terms $\varrho(y^{\ep_1} x^{\ep_2})$
for $\ep_i = \pm 1$ has entries that are $\pm1$ in the first row and
monic linear polynomials (up to sign) in the second row.
A simple inductive proof then shows that
$$\varrho(w)=\left[\begin{array}{cc} a(t)&b(t)\\ c(t) & d(t) \end{array} \right],$$
where $a(t), b(t), c(t), d(t)$ are monic polynomials in $t$ (up to sign) with $\deg a = \deg b = (\al-3)/2$
and $\deg c= \deg d = (\al-1)/2.$

Then
$$\varrho(xw)=\left[ \begin{array}{cc} \pm 1 & 1 \\ 0 & \pm 1 \end{array} \right]
\left[ \begin{array}{cc} a & b \\ c & d
\end{array} \right] =
\left[ \begin{array}{cc} \pm  a +c & \pm  b +d \\ \pm c &
\pm d
\end{array} \right]$$
and
$$ \varrho(wy)=
\left[ \begin{array}{cc} a &b \\ c  & d
\end{array} \right] \left[ \begin{array}{cc} \pm 1 & 0 \\ t & \pm 1
\end{array} \right]
=\left[ \begin{array}{cc}
\pm a + t b & \pm b \\
\pm c + t d &  \pm d
\end{array} \right].$$

Note that $\varrho(xw)=\varrho(wy)$ implies that $d=0$ and $c=tb$, thus
$$\varrho(w)= \left[\begin{array}{cc}a & b
\\ tb & 0
\end{array}\right].$$
Note that $d(t)=0$ shows $t$ is an algebraic integer.
Further,  $t b^2=-1$ since $\varrho(w) \in \SLC$.

Forming the product for the reversed word $w^*$
gives exactly the same matrix as one gets by flipping the matrix $\varrho(w)$ along the
anti-diagonal, thus
 $$\varrho(w^*)= \left[\begin{array}{cc}0 & b
\\ tb & a
\end{array}\right].$$  Hence
$$
 \varrho(ww^*)   =  \left[\begin{array}{cc} a   & b \\
tb & 0 \end{array}\right]\left[\begin{array}{cc} 0 & b \\
tb & a \end{array}\right]
    =  \left[\begin{array}{cc} tb^2 & 2ab \\ 0 & tb^2
   \end{array} \right]
   =    \left[\begin{array}{cc} -1 & 2ab \\ 0 & -1
   \end{array} \right]
   $$
since $tb^2=-1.$ Furthermore,
$$\varrho(x^{2n} ww^*)   =  \left[\begin{array}{cc}\pm 1 & 1 \\
 0 & \pm 1 \end{array} \right]^{2n}
 \left[\begin{array}{cc} -1 & 2 ab \\ 0 & -1
   \end{array} \right]
=\left[\begin{array}{cc} -1 & 2 a b \mp 2n  \\ 0 & -1
   \end{array} \right].$$
In summary, we have
$$\varrho(\mer)=\left[\begin{array}{cc} \pm1 & 1 \\ 0 & \pm 1
   \end{array} \right] \quad \text{and} \quad \varrho(\lng)=\left[\begin{array}{cc} -1 & 2ab \mp 2n \\ 0 & -1
   \end{array} \right].$$

If $\chi(\mer) =2$, then
\begin{align*}
\varrho(\mer^p \lng ^q) &= \left[\begin{array}{cc} 1 & 1 \\ 0 & 1
   \end{array} \right]^p \left[\begin{array}{cc} -1 &2ab - 2n \\ 0 & -1
   \end{array} \right]^q \\
   & = (-1)^q \left[\begin{array}{cc} 1 & p-q(2ab -2n) \\ 0 &  1
   \end{array} \right].
   \end{align*}
   Now $p$ and $q$ are assumed to be relatively prime, but for this matrix to equal
   the identity, we must have $q$ even and $p=q(2ab - 2n)$.
  Reducing this equation mod 2 and noting that $a=a(t), b=b(t)$ are polynomials over $\ZZ$ in
  an algebraic integer,  we conclude that $p$ must also be even,
   which is a contradiction.

   If $\chi(\mer) =-2,$ then
\begin{align*}
\varrho(\mer^p \lng ^q) &=  \left[\begin{array}{cc} -1 & 1 \\ 0 & -1
   \end{array} \right]^p \left[\begin{array}{cc} -1 &  2ab + 2n \\ 0 & -1
   \end{array} \right]^q \\
   &= (-1)^{p+q} \left[\begin{array}{cc} 1 & -p-q(2ab + 2n) \\ 0 &  1
   \end{array} \right].
   \end{align*}
For this matrix to equal the identity, we must have $p+q$ even and $p=-q(2ab + 2n)$.
Again, reducing mod 2, this shows that $p$ is even, which gives the desired contradiction.
\end{proof}

\subsection{Weights and algebraic multiplicities}
In this section we will show that the weights $m_i$ appearing in   Theorem \ref{surg-form} for
surgeries on a two-bridge knot are all equal to 1. This is achieved by identifying $m_i$ with the algebraic multiplicity $e_i$ of the
associated curve in the character variety $X(M)$ of the two-bridge knot complement..

In what follows, we will make use of continued fraction expansions  for $\be/\al$,
which are expressions of the form $\bn=[n_1, n_2, \ldots, n_k],$ where $n_i \in\ZZ$ and satisfy
$$\frac{\be}{\al} =  \frac{1}{n_1+ \frac{1}{\begin{array}{ccc} {n_2 \, +} \\ & \ddots \\ && +\, \frac{1}{n_k} \end{array}}}.$$
Note that, by replacing $[n_1, n_2, \ldots, n_k]$ with $[n_1, n_2, \ldots ,n_k \pm 1,\mp 1]$, we can always arrange $k$ to be odd.
If $K(\al,\be)$ is a two-bridge knot with continued fraction expansion $[n_1, n_2, \ldots, n_k]$, then $K$ can be drawn as the 4-plat closure of the braid $\si_2^{n_1} \si_1^{-n_2} \cdots \si_2^{n_k}$ as shown in Figure \ref{tunnel}.

Following Riley  \cite{Riley} and Le \cite{Le}, we explain how to view the character variety of a two-bridge knot group as a plane curve.
If $\Gamma$ is a group generated by the two elements $x$ and $y$,  then the Cayley-Hamilton theorem shows that every representation $\varrho \colon \Gamma \to \SLC$ is determined up to conjugation by the three traces
$\tr \varrho(x), \tr \varrho(y)$ and $\tr \varrho(xy)$.
In the case of a two-bridge knot group $G(\al,\be)$, the presentation \eqref{knotgroup} implies that $x$ and $y$ are conjugate, and using
the coordinates $t_1 =\tr \varrho(x) = \tr \varrho(y) $ and $t_2 =\tr \varrho(xy)$,
we can view the character variety as a plane algebraic curve in $\CC^2$.
Any plane curve is defined by a single polynomial $\Phi(t_1,t_2)$, which can be written in factored form as $\Phi=\prod_i \Phi_i^{e_i}$, where $\Phi_i$ are the irreducible components and $e_i  \in \ZZ^{>0}$ are their algebraic multiplicities.


\begin{proposition} \label{weights}
If $K$ is a two-bridge knot with complement $M = S^3 \sm \tau(K)$,
then each of the weights $m_i$ in the surgery formula of Theorem \ref{surg-form} is equal to the algebraic multiplicity $e_i$ of the associated
curve $X_i$ of $X(M)$. In particular, we have $m_i=1$.
\end{proposition}

\begin{proof}
Suppose $K = K(\al,\be)$ and let $M_{p/q}$ be the result of $(p/q)$-Dehn surgery along $K$.
Using an unknotting tunnel $T$, we will construct a genus two Heegaard splitting of $M$ and $M_{p/q}$,
which we use to identify the weight $m_i$ with the algebraic multiplicity $e_i$ for each curve $X_i$ of $X(M)$.

Choose a continued fraction expansion $[n_1, n_2, \ldots, n_k]$ for $\be/\al$ with $k$ is odd, and draw $K$ as the 4-plat knot
 as shown in Figure \ref{tunnel}. The arc $T$ connecting the bottom two lobes of $K$ is an unknotting tunnel for $K$, and attaching $T$ to $K$ determines  a genus two Heegaard splitting $(W_1, W_2)$ for $S^3$ with $W_1$ a regular neighborhood of $K \cup T$ (cf. Figure \ref{tunnel}).
Let $F= W_1 \cap W_2$ denote the splitting surface, which has genus $g(F)=2,$ and suppose $u$ is a closed curve in $F$ giving a meridian around
$T$ as shown. Let $v_{p/q}$ be a curve representing $\mer^p \lng^q$, chosen to avoid the attaching disk for the tunnel $T$.

\begin{figure}[ht]
\begin{center}
\leavevmode\hbox{}
\includegraphics[scale=0.80]{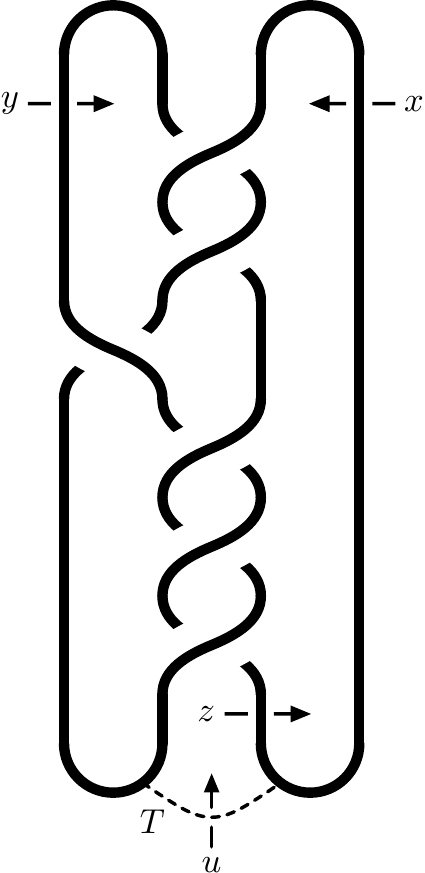}
\caption{The two-bridge knot $6_2=K(11,4)$ with continued fractions $[2,1,3]$ and its associated unknotting tunnel $T$} \label{tunnel}
\end{center}
\end{figure}

If $W_1' = W_1 \sm \tau(K)$ denotes the result of removing a small tubular neighborhood of $K$ from $W_1$, then  $W_1'$ and  $W_2$ form the knot complement $M$. Similarly, if $W_{p/q}$ is the result of attaching a disk to $W_1'$ along the curve $\nu_{p/q}$, then $(W_{p/q},W_2)$ is a Heegaard splitting for the closed 3-manifold $M_{p/q}.$

We have the following descriptions of character varieties:
\begin{eqnarray*}
X(W'_1) &=& \{\chi_{\varrho} \in X (F) \mid  \rho (u) = I \} \\
X (W_{p/q}) &=& \{ \chi_{\varrho}\in X (F) \mid \text{$\varrho (v_{p/q}) = I $ and $\varrho(u) = I$} \}.
\end{eqnarray*}
The Casson invariant
$\la_\SLC(M_{p/q})$ is by definition the intersection number of the zero-dimensional part of the intersection of  $X^*(W_{p/q})$ and  $X^*(W_2)$.

Suppose now that $X_i$ is a curve in $X (M)$. In order to compute the weight $m_i$ associated to $X_i$, we must view $X_i$ as a curve in the intersection of $X(W_1')$ and $X^*(W_2)$. This is the first step in computing the intersection of $X^*(W_{p/q})$ and $X^*(W_2)$. As explained in the proof of Proposition 4.3 of \cite{C}, the weight $m_i$ given in Theorem \ref{surg-form}  is
the intersection multiplicity of $X_i$ as a curve in the intersection $X^*(W_1') \cdot X^*(W_2)$ in the character variety of the splitting surface.

Let $x$ and $y$ be the two meridians of the knot shown, and choose curves $s$ and $t$ bounding disks in $W_2$ such that $x, y, s$ and $t$ form a symplectic basis for $H_1(F)$.
Thus $X(W_2) = \{ \chi_\varrho \in X(F) \mid \varrho(s) =I=\varrho(t) \},$ and it follows that  $X(W_2) \subset X(F)$ is homeomorphic to the character variety of the free group on two generators $x$ and $y$.

Consider the presentation of the knot group $\pi_1 M = G(\al,\be)  = \langle x,y \mid xw=wy\rangle$
from Equation \eqref{knotgroup},  where $w=y^{\ep_1}x^{\ep_2} \cdots y^{\ep_{\al-2}} x^{\ep_{\al-1}}$
for $\ep_i = (-1)^{\lfloor {i \be}/{\al}\rfloor}$. Notice that the generators $G(\al,\be)$ coincide with the generators of the free group
$\pi_1(W_2)$. The element $z$ shown in Figure \ref{tunnel} is given by $wyw^{-1}$, and it follows that $u$ represents the element $xz^{-1} = xwy^{-1}w^{-1}$, which is precisely the relation in the knot group $G(\al,\be)$.
Thus the defining equations for the curve $X_i$ in the intersection $X(W_1') \cdot
X^*(W_2)$ are precisely the defining equations of the curve $X_i$ in $X(M)$. It follows that $m_i=e_i$, and this proves the first statement of the proposition.

To make the second conclusion, notice that Proposition 3.4.1 of \cite{Le} implies that the algebraic multiplicities all satisfy $e_i=1$.
Clearly, this implies $m_i=1$ and completes the proof of the proposition.
\end{proof}

\subsection{Culler--Shalen seminorm} \label{csnorm}
In this section we review Ohtsuki's computation of
the Culler--Shalen seminorm $\| \cdot \|$
for  the complements $M = S^3 \sm \tau(K)$ of two-bridge knots (cf. \cite{O}).

A key ingredient in Ohtsuki's computations are the results of Hatcher and Thurston \cite{HT} determining
the boundary slopes for all two-bridge knots. They showed that two-bridge knots have  integral boundary slopes, and they established a one-to-one correspondence between non-closed, incompressible, $\partial$-incompressible surfaces in the knot complement and continued fraction expansions $[n_1, n_2, \ldots, n_k]$ for  $\be/\al$ and for ${(\be - \al)}/{\al}$ such that $|n_i| \geq 2$ for all $i=1, \ldots, k.$ Given a continued fraction expansion $\bn= [n_1, n_2, \ldots,  n_k]$, we set  $n^+$ to be the number of entries in $[n_1, n_2, \ldots,  n_k]$ whose signs agree with the alternating pattern  $[+,-, \ldots,(-1)^{k-2}, (-1)^{k-1}]$, and we set $n^-$ to be the number of entries with the opposite sign. We also denote by $n^+_0$ and $n^-_0$ the corresponding numbers for the Seifert expansion of $\be/\al$, which is the unique continued fraction expansion   with only even entries. Then by Proposition 2 of \cite{HT}, the boundary slope of the incompressible surface associated to the continued fraction expansion $\bn= [n_1, n_2, \ldots, n_k]$ is given by
\begin{equation} \label{b-slope}
N_{\bn} = 2[(n^+ - n^-) - (n^+_0 - n^-_0)].
\end{equation}

The Culler--Shalen seminorm of a slope $p/q$ is a weighted sum of the intersections of $p/q$ with the boundary slopes of $K$. In Proposition 5.2 of this paper, Ohtsuki finds these weights for a function called $\Phi(p,q)$, with one minor error. In this proposition he shows that the coefficient of the intersection $|p - N_\bn q|$ for $\Phi(p,q)$ is $1/2 \prod_j (|n_j|-1)$. However in Section 4 he shows that the coefficient of $|p|$ in his formula should rather be $1/2 \prod_j (|n_j| - 1)-1/2$,
since the continued fraction expansion of $\be/\al$ for which all terms are even corresponds to the $1/0$--boundary slope. The following result is essentially Proposition 5.2 from \cite{O}, after making this  small correction to the statement
(cf.~p.~18 of Mattman's thesis \cite{M}).
\begin{proposition}[Ohtsuki]
$$\Phi(p,q) = \frac{1}{2}\left(-|p| + \sum_{\bn=[n_1,\ldots, n_k]} \left|p - N_\bn q\right| \; \prod_{j=1}^k \left(|n_j|-1\right)\right).$$
\end{proposition}

Ohtsuki's calculation adds the order of the poles
 of the functions $f_{i,\al}$ over ideal points on each curve of the character variety.
Therefore the total Culler--Shalen seminorm equals $\|p/q\|_{T} = \Phi(p,q)$.

\subsection{The correction terms}
In this subsection, we compute the correction terms that appear in the surgery formula (see Theorem \ref{surg-form}). These terms compensate for characters $\chi_\varrho$ of irreducible representations $\varrho \in R^*(M)$ which satisfy $$\varrho(\mer^p \lng^q) = \begin{pmatrix} 1 & 1 \\ 0 & 1 \end{pmatrix}.$$
Such representations satisfy $\tr \varrho(\mer^p \lng ^q)=2$, but they do not extend over  $(p/q)$--Dehn surgery on $K$. The correction terms depend only on the parity of $p$ and are denoted $E_0$ and $E_1$. Specifically, by the proof of Theorem 4.8 of \cite{C}, $E_0$ is the sum of the weights of characters in $X(M)$ for which $\chi(\mer) = \pm 2$,  $\chi(\lng)=\pm 2$, and $\chi(\mer^p \lng^q) = 2$ but for which there is no corresponding representation $\varrho$ of
$X(\Si_{p/q}(K))$, where $p$ is even, and $E_1$ is defined similarly in the case where $p$ is odd.

\begin{proposition}\label{corr}
Given the two-bridge knot $K(\al,\be)$, we have $E_0 = 0$ and $E_1 = (\al - 1)/4$.
\end{proposition}

\begin{proof}
By the proof of Proposition \ref{reg}, any irreducible $\SLC$ representation $\varrho$ of the knot group with $\chi_{\varrho}(\mer) = \pm 2$ satisfies $\chi_{\varrho}(\lng) = -2$. It follows that $\chi_{\varrho}(\mer^p \lng^q) = -2$ if $p$ is even, so $E_0 = 0$.

If $p$ is odd, we may compute $E_1$ by considering the result of $1/0$--surgery on $K$, which yields $S^3$. In this case $\lambda(M_{1/0}) = \lambda(S^3) = 0$, so in fact $E_1 = 1/2\| 1/0 \|_{T} = \Phi(1,0)/2$. But in Section 1 of \cite{O}, Ohtsuki shows that $\Phi(1,0) = (\al - 1)/2$, and this completes the proof.
\end{proof}

To summarize,
we deduce the following.
\begin{theorem} \label{two-b}
Let $K=K(\al,\be)$ be a two-bridge knot.
 Suppose $p/q$ is not a strict boundary slope, and suppose no $p'$-th root of unity is a root of the Alexander polynomial of
$K$, where $p'=p$ if $p$ is odd and $p' = p/2$ if $p$ is even. Then
$$\la_\SLC (M_{p/q} ) =  \begin{cases}
\frac{1}{2} \| p/q\|_{T} & \text{if $p$ is even,} \\
\frac{1}{2} \| p/q \|_{T} - (\al - 1)/4 & \text{if $p$ is odd.}
\end{cases} $$
\end{theorem}

Note that Ohtsuki's fomula allows Culler--Shalen seminorms to be computed for any two-bridge knot by hand on a case-by-case basis. The contributing boundary slopes and their weights are given for all prime knots up to 8 crossings in the appendix. Then Theorem \ref{two-b} allows the computation of the $\SLC$ Casson invariant for all admissible surgeries on two-bridge knots on a case-by-case basis.

\subsection{Double twist knots} \label{double twist}
We now compute general forumlas for the total Culler--Shalen seminorms for the double twist knots $J(\ell,m)$ depicted in Figure \ref{gentwist}. Here, $J(\ell,m)$ is drawn so that the vertical strands are twisted positively when $\ell >0$ and negatively when $\ell <0$ and so that the horizontal strands are twisted positively when $m>0$ and negatively when $m<0$. For example, $J(2,-2)$ is the figure-8 knot, $J(2,2)$ is the left-handed trefoil, and $J(-2,-2)$ is the right-handed trefoil.

Note that $J(0,m)$ and $J(\ell,0)$ are unknots, and that $J(\pm 1,m)$ and $J(\ell, \pm 1)$ are just $(2,p)$ torus knots or links.
In the general case,
by sliding the horizontal crossings up and to the left, we can view $J(\ell,m)$ as the two-bridge knot (or link) with continued fraction expansion $[\ell, -m]$ that is obtained as the 4-plat closure of the braid $\si_2^\ell \si_1^m.$
Notice that $J(\ell,m)$ is a knot whenever at least one of $\ell$ and $m$ is even.
Otherwise, if $\ell$ and $m$ are both odd, then $J(\ell,m)$ is a link with two components.

 By manipulating the knot diagram, one can easily show that $J(\ell,m) = J(m,\ell)$ and that the  $J(\ell,m)$ has mirror image given by $J(-\ell,-m)$. Noting that for any two-bridge knot $K$, the continued fraction expansions of the mirror image $\overline{K}$ are precisely the negatives of the continued fraction expansions of $K$, we see from the boundary slope formula of Hatcher--Thurston \cite{HT} that the boundary slopes of the mirror image $\overline{K}$  are the negatives of the boundary slopes of $K$ and from the weight formula of Ohtsuki \cite{O} that the weights of the corresponding boundary slopes for $K$ and its mirror image $\overline{K}$ are identical.
We therefore restrict our attention to knots $J(\ell,\pm m)$ with $\ell$ and $m$ positive. We 
consider separately the two cases $J(\ell,-m)$ and $J(\ell,m)$, where $\ell,m >0$.

\smallskip
\noindent \textbf{Case I:} $J(\ell,-m)$ for $\ell,m >0.$

Using the continued fraction expansion $[\ell, m]$, we
determine the rational number
$$\frac{\be}{\al} = \frac{1}{\ell+\frac{1}{m}}.$$
Setting  $\al=\ell m+1$ and $\be = m$, we see that
the double twist knot $J(\ell, -m)$ coincides with the two-bridge knot $K(\al,\be)$.
  In the special case $\ell =1,2$, the knots $J(1,-m)$ and $J(2,-m)$ specialize to the $(2,m+1)$-torus knots and the $m$-twist knots, respectively. When both  $\ell, m \geq 2$, the continued fraction expansion $[\ell,m]$ corresponds to an incompressible, $\partial$-incompressible surface  in the complement of the knot.

\begin{figure}[ht]
\begin{center}
\leavevmode\hbox{}
\includegraphics[scale=0.80]{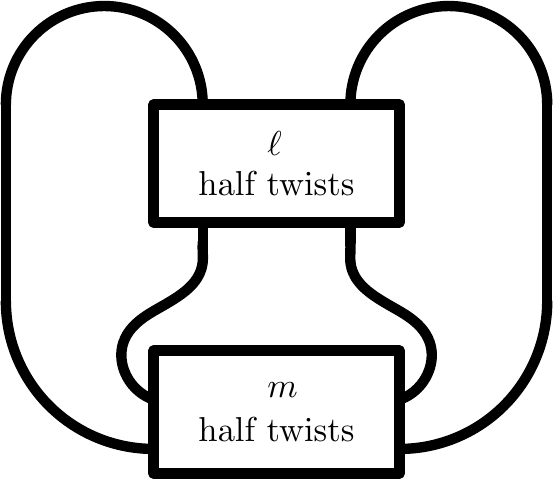}
\caption{The double twist knot $J(\ell,m)$} \label{gentwist}
\end{center}
\end{figure}

An arbitrary two-bridge knot $K=K(\al,\be)$ is a double twist knot of this form if $\al= \ell \be+1$ for some positive integer $\ell,$ and its mirror image, which is the two-bridge knot $\overline{K}=K(\al,-\be)$,  is a double twist knot of this form if $\al = \ell (\al - \be) + 1$ for some positive integer $\ell.$

To calculate the Culler--Shalen seminorms of $J(\ell,-m)$, we first determine all possible continued fraction expansions $[n_1,\ldots, n_k]$ for $\be/\al$ and for $\frac{\be - \al}{\al}$ with $|n_i|\ge 2$. For $\be/\al$,
because $\ell < \frac{\al}{\be} < \ell+1$, we see that $n_1$ must be either $\ell$ or $\ell+1$. If $n_1=\ell,$ the only possibility is  $[\ell,m]$. Otherwise, if $n_1=\ell+1,$ then the condition that $|n_j| \ge 2$ implies that the only possibility is the alternating expansion $[\ell+1,-2,2, \ldots,(-1)^{m-1} 2]$ of length $m$.

Because $-2 < \frac{\al}{\be - \al} < -1$, any continued fraction expansion $[n_1,\ldots, n_k]$ for $\frac{\be - \al}{\be}$ with $|n_i|\ge 2$ must begin with $n_1=-2,$ and the only possibility  is the alternating expansion $[-2,2,\ldots, (-1)^{\ell-1} 2, (-1)^\ell(m+1)]$ of length $\ell$.

Let $A$ denote the continued fraction expansion $[\ell,m]$, $B$ the continued fraction expansion $[\ell+1,-2, \ldots , (-1)^{m-1}2]$ of length $m$, and $C$ the continued fraction expansion $[-2,2, \ldots, (-1)^\ell (m+1)]$ of length $\ell$. For $\ell=1$ or $m=1$, because the expansion $A=[\ell,m]$ has entries with absolute value less than 2, it does not correspond to a boundary slope.

Recall that the Seifert expansion is the unique continued fraction expansion $[n_1,\ldots, n_k]$ with each $n_i$ even for $i=1,\ldots, k.$
Which of $A,B$ or $C$ is the Seifert expansion depends on the parities of $\ell$ and $m$. Since $K$ is a knot, at most one of $\ell$ or $m$ is odd, and one sees that the Seifert expansion is $A$ if both $\ell$ and $m$ are even, $B$  if $\ell$ is odd, and $C$ if $m$ is odd.

We now use Equation \eqref{b-slope} to calculate the boundary slopes for $A,B,$ and $C$.  For the continued fraction expansions  obtained above,  each of $n^+$ and $n^-$ can be computed directly, and we obtain that
 $n^+ - n^- = 0$ for $A$,  $n^+ - n^- = m$ for $B$, and $n^+ - n^- = -\ell$ for $C$.
The other two terms $n_0^+$ and $n_0^-$ depend on the parities of $\ell$ and $m$ and the difference $n_0^+-n_0^-$ will equal $n^+-n^-$ for $A, B$, or $C$ depending on whether  $\ell$ and $m$ are even, or $\ell$ is odd, or $m$ is odd, respectively.
The resulting values of the boundary slopes $N_\bn$ for each of $A, B$ and $C$ in these three separate cases are summarized in Table \ref{boundaryslope}.

\begin{table}[t] 
\begin{center}

 \begin{tabular}{|c|c|c|c|}
\hline $\bn=[n_1,\ldots, n_k]$& \multicolumn{3}{c|}{boundary slope $N_\bn$}\\ \cline{2-4}
continued fraction & $\ell, m$ even & $\ell$ odd & $m$ odd  \\ \hline
 $A=[\ell,m]$ & $0$ & $-2m$ & $2 \ell$  \\
 $B= [\ell+1,-2, \ldots ,(-1)^{m-1} 2]$ & $2m$ & $0$ & $2(\ell +m)$ \\
 $C=[-2,2, \ldots, (-1)^\ell(m+1)]$ & $-2\ell$ & $-2(\ell+m)$ & $0$ \\ \hline
 \end{tabular}
 \end{center}
 \caption{The boundary slopes for  $J(\ell,-m)$ } \label{boundaryslope}
 \end{table}

In order to deduce the Culler--Shalen seminorms, we also need to compute the weights of $\bn=[n_1\ldots, n_k].$ These are given by $ \frac{1}{2} \prod_j (|n_j|-1)$ for all slopes other than the 0-slope, which has weight  $ \frac{1}{2} \prod_j (|n_j|-1) - 1/2$. The values $\prod_j (|n_j|-1)$ are easily determined and are independent of the parities of $\ell$ and $m$; the results are summarized in Table \ref{weightsA}.

\begin{table}[ht]  
\begin{center}
 \begin{tabular}{|c|c|}
\hline $\bn=[n_1,\ldots, n_k]$ & $ \prod_j (|n_j|-1)$ \\ \hline
 $A=[\ell,m]$ &  $(\ell-1)(m-1)$ \\
 $B= [\ell+1,-2, \ldots ,(-1)^{m-1} 2]$  & $\ell$ \\
 $C=[-2,2,\ldots,(-1)^\ell (m+1)]$ &   $m$  \\ \hline
 \end{tabular} \end{center}  
 \caption{Computing the weights for $J(\ell,-m)$ } \label{weightsA}
 \end{table}

 The next theorem
  summarizes this discussion and presents a computation of seminorms for $J(\ell,-m)$.
\begin{theorem} \label{CS-formulaI} If $K=J(\ell,-m)$ with $\ell,m \geq 1$ is a double twist knot, then the Culler--Shalen seminorm of $p/q$ is given by
 $$ 2 \| p/q \|_{T}=  \\ 
 \begin{cases}
\begin{split}
(\ell m &- \ell -m)|p| + \ell|p - 2mq|  \\
& \!\!\!\! + m|p+2\ell q|
\end{split} & \text{if $\ell, m$ are even,} \vspace{3mm}  \\
\begin{split}
(\ell&-1)(m-1)|p+2mq| + (\ell-1)|p| \\
&+ m|p+2(\ell+m)q|\end{split}
& \text{if $\ell$ is odd,} \vspace{3mm}  \\
\begin{split}
(\ell &-1)(m-1)|p-2\ell q|  \\
&+\ell|p-2(\ell+m)q|+(m-1)|p| \end{split} & \text{if $m$ is odd.}
\end{cases} $$
\end{theorem}
Note that in case $\ell=1$ or $m=1$, $J(\ell,-m)$ is a torus knot and $\| \cdot \|_T$ is not a norm.
In this case, $A$ does not correspond to a boundary slope, but the corresponding weight is $(\ell-1)(m-1) =0$ and so the formula in Theorem \ref{CS-formulaI} remains valid.
Otherwise, if $\ell, m \geq 2,$ then $\| \cdot \|_T$ is a norm.

\smallskip
\noindent \textbf{Case II:} $J(\ell,m)$ for $\ell,m >0.$

Note that $J(1,m) =  J(m,1)$ is the $(2,m-1)$-torus knot, which is already covered in Case I as the mirror image of $J(1,2-m)$. Thus, we may assume throughout that $\ell,m \geq 2$. Further, since $J(\ell,m) = J(m,\ell)$, we can assume $\ell \ge m$.

Using the continued fraction expansion $[\ell, -m]$, we determine the rational number
$$\frac{\be}{\al} = \frac{1}{\ell-\frac{1}{m}}.$$
Setting  $\al=\ell m-1$ and $\be = m$, we see that
the double twist knot $J(\ell,m)$ coincides with the two-bridge knot $K(\al,\be)$. Since both  $\ell, m \geq 2$, the continued fraction expansion $[\ell,-m]$ corresponds to an incompressible, $\partial$-incompressible surface  in the complement of the knot.

An arbitrary two-bridge knot $K=K(\al,\be)$ is a double twist knot of this form if $\al= \ell \be-1$ for some positive integer $\ell.$ Its mirror image $\overline{K}$ is a double twist knot of this form if $\al = \ell (\al - \be) - 1$ for some positive integer $\ell.$

To calculate the Culler--Shalen seminorms of $J(\ell,m)$, we must determine all possible continued fraction expansions $[n_1,\ldots, n_k]$ for $\be/\al$ and for $\frac{\be - \al}{\al}$ with $|n_i|\ge 2$.
For $\be/\al$, because $\ell - 1 < \frac{\al}{\be} < \ell$, we see that $n_1$ must be either $\ell$ or $\ell-1$. If $n_1=\ell,$ the only possibility is  $[\ell,-m]$. Otherwise, if $n_1=\ell-1,$ then the condition that $|n_j| \ge 2$ implies that the only possibility is the expansion $[\ell-1,2,-2, \ldots,(-1)^{m} 2]$ of length $m$, which is alternating beginning with the second term.

Because $-2 < \frac{\al}{\be - \al} < -1$, any continued fraction expansion $[n_1,\ldots, n_k]$ for $\frac{\be - \al}{\be}$ with $|n_i|\ge 2$ must begin with $n_1=-2.$ Here there are two possibilities: the expansion $[-2,2,\ldots, (-1)^{\ell-1} 2, (-1)^{\ell - 1}(m-1)]$ of length $\ell$, which is alternating except for the last term, and the alternating expansion
$$[-2,2,\ldots, (-1)^{\ell} 2, (-1)^{\ell - 1} 3, (-1)^{\ell} 2, \ldots, (-1)^{\ell + m -1} 2]$$ of length $\ell + m - 3,$ where $(-1)^{\ell - 1} 3$ is the $(\ell - 1)$-st term.

Let $A$ denote the continued fraction expansion $[\ell,-m]$, $B$ the continued fraction expansion $[\ell-1, 2, \ldots , (-1)^{m}2]$ of length $m$, $C$ the continued fraction expansion $[-2,2, \ldots, (-1)^{\ell - 1}2, (-1)^{\ell-1} (m-1)]$ of length $\ell$, and $D$ the continued fraction expansion  $$[-2,2,\ldots, (-1)^{\ell} 2, (-1)^{\ell - 1} 3, (-1)^{\ell} 2, \ldots, (-1)^{\ell + m -1} 2]$$ of length $\ell + m - 3$.

For $\ell =2$, because the expansion $B=[\ell-1, 2, \ldots , (-1)^{m}2]$ has an entry with absolute value less than 2, it does not correspond to a boundary slope. Likewise, for $m =2$,   $C= [-2, \ldots, (-1)^{\ell-1} (m-1)]$  has an entry with absolute value less than 2 and does not correspond to a boundary slope.

As before, which of $A,B$, or $C$ is the Seifert expansion depends on the parities of $\ell$ and $m$. (Note that $D$ is never the Seifert expansion.)
Indeed, just as in the previous case, the Seifert expansion is $A$ if both $\ell$ and $m$ are even, $B$  if $\ell$ is odd, and $C$ if $m$ is odd.

We now use Equation \eqref{b-slope} to calculate the boundary slopes for $A,B,$ $C,$ and $D$.  We see that
 $n^+ - n^- = 2$ for $A$,  $n^+ - n^- = 2 - m$ for $B$,  $n^+ - n^- = 2 -\ell$ for $C$, and $n^+ - n^- = 3 - \ell - m$ for $D$.
The resulting values of the boundary slopes $N_\bn$ for each of $A, B, C$ and $D$ for the three separate cases are summarized in Table \ref{boundaryslope2}.

\begin{table}[th]  
\begin{center}
 \begin{tabular}{|c|c|c|c|}
\hline $\bn=[n_1,\ldots, n_k]$& \multicolumn{3}{c|}{boundary slope $N_\bn$}\\ \cline{2-4}
continued fraction & $\ell, m$ even & $\ell$ odd & $m$ odd  \\ \hline
 $A=[\ell,-m]$ & $0$ & $2m$ & $2\ell$  \\
 $B= [\ell-1,2, \ldots ,(-1)^{m} 2]$ & $-2m$ & $0$ & $2(\ell - m) $ \\
 $C=[-2,2, \ldots, (-1)^{\ell-1}(m-1)]$ & $-2\ell$ & $2(m - \ell)$ & $0$ \\
 $D=[-2,2,\ldots, (-1)^{\ell - 1} 3, \ldots, (-1)^{\ell + m -1} 2]$ & $2(1 - \ell - m)$ & $2(1 - \ell)$ & $2(1 - m)$ \\ \hline
 \end{tabular} \end{center} 
 \caption{The boundary slopes for $J(\ell,m)$ } \label{boundaryslope2}
 \end{table}

The weights of $\bn=[n_1\ldots, n_k]$ for the Culler--Shalen seminorms are computed as in Case 1. The values $\prod_j (|n_j|-1)$ are once again easily determined and independent of the parities of $\ell$ and $m$; the results are summarized in Table \ref{weights2}.

\begin{table}[th] 
\begin{center}
 \begin{tabular}{|c|c|}
\hline $\bn=[n_1,\ldots, n_k]$ & $ \prod_j (|n_j|-1)$ \\ \hline
 $A=[\ell,-m]$ &  $(\ell-1)(m-1)$ \\
 $B= [\ell-1,2,-2, \ldots ,(-1)^{m} 2]$  & $\ell-2$ \\
 $C=[-2,2,\ldots,(-1)^{\ell-1} (m-1)]$ &   $m - 2$ \\
 $D=[-2,2,\ldots, (-1)^{\ell - 1} 3, \ldots, (-1)^{\ell + m -1} 2]$ & $2$ \\ \hline
 \end{tabular} \end{center} 
 \caption{Computing the weights for $J(\ell,m)$ } \label{weights2}
 \end{table}

 The next theorem
  summarizes this discussion and presents a computation of seminorms for $J(\ell,m)$.
\begin{theorem} \label{CS-formulaII} If $K=J(\ell,m)$ with $\ell,m \geq 2$ is a double twist knot, then the Culler--Shalen seminorm of $p/q$ is given by
 $$ 2\| p/q \|_{T}=  \\
\begin{cases}
\begin{split}
(\ell m &- \ell -m)|p| + (\ell - 2)|p + 2mq| \\
& \!\!\!\! + (m - 2)|p + 2\ell q| + 2 |p + 2(\ell + m-1)q|
\end{split} & \text{if $\ell, m$ are even,}  \vspace{3mm}  \\
\begin{split}
(\ell&-1)(m-1)|p - 2mq| + (\ell-3)|p| \\
&+ (m - 2)|p + 2(\ell - m)q| + 2 |p+2(\ell - 1)q|
\end{split}
& \text{if $\ell$ is odd,} \vspace{3mm}  \\
\begin{split}
(\ell &-1)(m-1)|p-2\ell q| +(\ell - 2)|p+2(m-\ell)q| \\
& +(m-3)|p| + 2|p + 2(m-1)q|
\end{split} & \text{if $m$ is odd.}
\end{cases} $$
 \end{theorem}

Note that $\| \cdot \|_T$ is a norm unless $\ell=2=m.$
In the case $\ell=2$, $B$ does not correspond to a boundary slope, but since the corresponding weight is $\ell-2$, which evaluates to zero, the formula in Theorem \ref{CS-formulaII} remains valid.
Similarly, in the case $m=2$, $C$ does not correspond to a boundary slope, but since the corresponding weight is $m-2$, which evaluates to zero,  the formula remains valid.

Finally, since $J(2,m) = J(-2,m-1)$ and $J(\ell, 2) = J(\ell-1, -2)$, we see these knots (or their mirror images) are also covered in Case I.

These computations may be combined with Theorem \ref{two-b} to compute $\la_\SLC(M_{p/q})$ for manifolds which result from admissible surgeries on double twist knots.

\begin{corollary} \label{doubletwist}
Let $J = J(\ell, \pm m)$ be a double twist knot, where $\ell$ and $m$ are positive.
 Suppose $p/q$ is not a strict boundary slope, and suppose no $p'$-th root of unity is a root of the Alexander polynomial of
$K$, where $p'=p$ if $p$ is odd and $p' = p/2$ if $p$ is even. Then
$$\la_\SLC (M_{p/q} ) =  \begin{cases}
\frac{1}{2} \| p/q\|_{T} & \text{if $p$ is even,} \\
\frac{1}{2} \| p/q \|_{T} - 1/4 (\ell m  \mp 1 -1) & \text{if $p$ is odd.}
\end{cases} $$
Here $\|p/q\|_T$ is given by Theorem \ref{CS-formulaI} for $J(\ell,-m)$ and by Theorem \ref{CS-formulaII} for $J(\ell,m)$.
\end{corollary}

\subsection{A nontriviality result}
In this subsection, we prove that the $\SLC$ Casson invariant is nontrivial for nearly all $p/q$ surgeries on a two-bridge knot.
Our proof makes essential use of the classifications of exceptional surgeries on two-bridge knots given by Brittenham and Wu \cite{BW}.

This raises the question of whether, and under what circumstances, the $\SLC$ Casson invariant $\la_\SLC(\Si)$ will vanish for a given 3-manifold $\Si$. This will occur for instance whenever
$X^*(\Si)$ has no zero-dimensional components.
For example, if $\pi_1(\Si)$ is cyclic, then $X^*(\Si) = \varnothing$ and so $\la_\SLC(\Si)=0$ follows immediately. This shows that any
cyclic surgery on a knot $K$ necessarily has trivial $\SLC$ Casson invariant.
However, Theorem 18.2  in \cite{BZ2} shows that hyperbolic two-bridge knots do
not admit finite surgeries. In particular, they do not admit any cyclic surgeries, so one cannot produce examples of this type under Dehn surgery on a two-bridge knot.

Using the operation of spliced sum, one can construct homology 3-spheres $\Si$ whose character varieties $X^*(\Si)$ are nonempty but do not contain any zero-dimensional components (see \cite{BC2} for more details). These examples have $\la_\SLC(\Si) =0$ and include 3-manifolds obtained by $(-1)$ Dehn surgery on the untwisted Whitehead double of a knot $K$ in $S^3$.
It is an interesting question whether one can find a knot for which many (or most) Dehn surgeries have $\la_\SLC(\Si)=0$. In what follows, we show that such examples are not to be found among two-bridge knots.

Suppose now $K$ is a two-bridge knot, and let $M =S^3 \sm \tau(K)$ denote its complement and $M_{p/q}$ the closed 3-manifold obtained by $(p/q)$--Dehn surgery on $K$.

We define a subset $\cE_K \subset \ZZ$ of the set of exceptional slopes for $K$ as follows.
For $K\neq J(\ell,\pm m)$, set $\cE_K = \varnothing.$
For $K = J(\ell, \pm m)$ with $\ell, m >2$,   set $\cE_K = \{0\}$ if $\ell, m$ are both even and set $\cE_K = \{2m\}$ if $\ell$ is odd and $m$ is even.
For $\ell > 2$ and $K = J(\ell, 2)$, set $\cE_K = \{- 4\}$. For $K = J(\ell, -2)$, set $\cE_K = \{ 4\}$.
Finally, for the figure eight knot $K=J(2,-2)$, set $\cE_K = \{ 4, -4\}.$
Note that $\cE_K$ has at most two slopes, and apart from the figure eight knot, $|\cE_K | \leq 1.$
 \begin{theorem} \label{nontriv}
Suppose $K$ is a hyperbolic two-bridge knot in $S^3$. If $p/q \neq 1/0$ with   $p/q \not\in \cE_{K}$, then
$\la_\SLC(M_{p/q}) > 0$.

 \end{theorem}

 \begin{proof}
We recall the classification of exceptional surgeries on two-bridge knots given by Theorem 1.1 in \cite{BW}, and we note that as a consequence of Thurston's orbifold theorem \cite{BLP}, every non-exceptional surgery is actually hyperbolic.
We note further that, in the case of a hyperbolic surgery, positivity of the $\SLC$ Casson invariant follows directly from
Proposition 3.2 of \cite{C}, though one can say more (see Remark \ref{remhyp}).

Using Theorem 1.1 \cite{BW}, we see that apart from the double twist knots considered previously, any other hyperbolic two-bridge knot admits only hyperbolic surgeries.
The double twist knots $J(\ell, \pm m)$ split into three possible cases, as follows.

\smallskip\noindent
 {Case I. $J(\ell,\pm m), \quad \ell, m >2.$}

In this case, there is exactly one exceptional surgery $\ga$. If $\ell$ and $m$ are both even, then $\ga =0$ and is a strict boundary slope (unless $J(\ell,\pm m)$ is fibered). \
 Otherwise, if $\ell$ is odd and $m$ is even, then $\ga = 2m$ and is a strict boundary slope. Note that Theorem \ref{surg-form} does not apply in either case.

\smallskip\noindent
 {Case II. $J(\ell,\pm 2), \quad \ell$ even.}

In this case, there are five exceptional surgeries: $0, \mp 1, \mp 2, \mp 3, \mp 4.$ Consider first $J(\ell,-2)$. For each $\ga \in \{0, 1, 2, 3\}$ the surgery formula applies, and Theorems \ref{CS-formulaI} and \ref{two-b} show that $\la_\SLC(M_\ga) >0.$ The remaining slope $\ga = 4$ is a  strict boundary slope, so Theorem \ref{surg-form} does not apply in this case.

The case of   $J(\ell,2)$ is very similar; for each $\ga \in \{0, -1, -2, -3\}$ the surgery formula applies, and Theorems \ref{CS-formulaII} and \ref{two-b} show that $\la_\SLC(M_\ga) >0.$ The remaining slope $\ga = -4$ is a  strict boundary slope, so Theorem \ref{surg-form} does not apply in this case.

\smallskip\noindent
 {Case III. $J(2, -2).$}

This is the figure eight knot, which has nine exceptional surgeries: $-4, -3, -2,$ $-1, 0,1, 2,  3, 4.$
For each $\ga \in \{0, \pm 1, \pm 2, \pm 3\},$ the surgery formula applies, and Theorems \ref{CS-formulaI} and \ref{two-b} show that $\la_\SLC(M_\ga) >0.$
The remaining two slopes $\{\pm 4\}$ are both strict boundary slopes and  Theorem \ref{surg-form} does not apply.
\end{proof}

\begin{remark} \label{remhyp}
In the case of a hyperbolic surgery,  we can use Mostow rigidity to see that the character $
\chi_\varrho$ of any representation lifting the discrete faithful $PSL(2,\CC)$ representation coming from the hyperbolic structure is an isolated smooth point of the character variety.
Theorem 2.1 in \cite{BC1} then applies to show $\chi_\varrho$ has intersection multiplicity $1$ and  contributes $1$ to the $\SLC$ Casson invariant. As a consequence, we see that, in terms of the surgery formula (Theorem \ref{surg-form}),  the component $X_i$ of $X(M)$ containing $\chi_\varrho$ must have weight $m_i=1.$

This has the following interesting consequence: if $p/q$ is a hyperbolic Dehn surgery on a hyperbolic knot $K$, any one-dimensional component $X_i$ of $X(M)$ with weight $m_i \neq 1$ cannot contain the character of a discrete faithful representation $\varrho\colon \pi_1(M_{p/q}) \to \SLC$ lifting the hyperbolic representation $\varrho_0\colon \pi_1(M_{p/q}) \to {PSL}(2,\CC).$
\end{remark}

\section{The $A$-polynomial and the $\widehat{A}$-polynomial} \label{A-polynomial}
In this section, we recall the definitions of the $A$-polynomial and $\widehat{A}$-polynomial, and we establish results on their $M$-degree and $L$-degree for two-bridge knots.
We use this to identify certain two-bridge knots $K$ for which the $A_K(M,L) \neq \widehat{A}_K(M,L)$,
and we illustrate how to use the Culler--Shalen seminorm data to determine $\widehat{A}_K(M,L)$ in the specific cases
of the knots $7_4$ and $8_{11}.$ We end the section by examining the relationship between
the $\SLC$ knot invariant $\la'_{\SLC}(K)$
and the $M$-degree of $\widehat{A}_{K}(M,L).$

\subsection{Degrees of the $\widehat{A}$-polynomial}
We begin by recalling the definition of the $A$-polynomial $A_K(M,L)$  from \cite{CCGLS} (see also \cite{CL1,CL2}). Given a knot
$K$ in $S^3$, let $M = S^3 \sm \tau(K)$ be its complement and choose  a standard meridian-longitude pair $(\mer,\lng)$ for $\pi_1(\partial M)$. Setting
$$\La = \{ \varrho\colon\pi_1(\partial M) \lto \SLC \mid \text{$\varrho(\mer)$ and  $\varrho(\lng)$ are diagonal matrices} \},$$
notice that the eigenvalue map  $\La \lto \CC^* \times \CC^*$, which is defined by setting $\varrho \mapsto (u,v) \in \CC^* \times \CC^*$, where $$\varrho(\mer) = \begin{pmatrix}u& 0 \\ 0 &u^{-1}\end{pmatrix} \text{ and }\varrho(\lng) = \begin{pmatrix}v& 0 \\ 0 &v^{-1}\end{pmatrix},$$
 identifies $\La$ with  $\CC^* \times \CC^*$ and that the natural projection $p\colon \La \lto X(\partial M)$ is a degree 2, surjective, regular map.

The natural inclusion $\pi_1(\partial M) \lto \pi_1(M)$ induces a map $r\colon X(M) \lto X(\partial M)$, which is regular. We define $V \subset X(\partial M)$ to be the Zariski closure of the union of the image $r(X_i)$ over each component $X_i \subset X(M)$ for which $r(X_i)$ is one-dimensional, and we set
$D \subset \CC^2 $ to be the Zariski closure of the algebraic curve $p^{-1}(V) \subset \La,$ where we  identify $\La$ and  $\CC^* \times \CC^*$ via the eigenvalue map.
The $A$-polynomial $A_K(M,L)$ is just the defining polynomial of the plane curve $D \subset \CC^2$;
it is well-defined up to sign by requiring it to have integer coefficients with greatest common divisor one. By convention, we remove the factor $L-1$ associated to the reducible representations in $A_K(M,L)$ and also delete any repeated factors.

In  \cite{BZ2}, Boyer and Zhang define an $A$-polynomial $A_{X_i}(M,L)$ for each one-dimensional component $X_i$ of $X(M)$ that is a norm curve (see Definition \ref{CS-seminorm}). Their approach takes into account repeated factors by using the degree of the restriction $r|_{X_i}$ as the multiplicity. Although the definition of $A_{X_i}(M,L)$ in \cite{BZ2} assumes $X_i$ is a norm curve, the approach works for any one-dimensional component $X_i$ of $X(M)$ such that $r(X_i)$ is one-dimensional \cite{BZ0}.
 Taking the product
$$\widehat{A}_K(M,L) =  A_{X_1}(M,L) \cdots A_{X_n}(M,L)$$ over all one-dimensional components $X_i$ of $X(M)$ such that $r(X_i)$ is one-dimensional, we obtain an alternative version of the $A$-polynomial  that includes factors with multiplicities. For small knots, it is not difficult to check that $A_K(M,L)$ and $\widehat{A}_K(M,L) $ have the same factors, only that $\widehat{A}_K(M,L)$ may include some repeated factors.

For each such curve $X_i$, the Culler--Shalen seminorm and the polynomial ${A}_{X_i}(M,L)$ are intimately related. In fact the Culler--Shalen seminorm  $\| \cdot \|_{i}$  is precisely the width function norm determined by the Newton polygon of ${A}_{X_i}(M,L)$ by Proposition 8.8 of \cite{BZ2}.  (Here note that the total Culler---Shalen seminorm defined in this paper is half that of \cite{BZ2}.) In particular, for a curve $X_i$ component of $X(M),$ the $A$-polynomial $A_{X_i}(M,L)$ has $L$-degree given by $\|1/0\|_{i}$ and $M$-degree given by $\| 0/1 \|_{i}$ (Proposition 6.6, \cite{BZ2}).
Since $\widehat{A}_K(M,L)$ is a product over components $X_i$ in $X(M),$ we see that for two-bridge knots the $L$- and $M$-degrees of $\widehat{A}_K(M,L)$  are given by the total Culler---Shalen seminorm as $ \|1/0 \|_{T}$ and $ \| 0/1 \|_{T}$, respectively. (This last step makes use of the result in Proposition \ref{weights} saying that the weights $m_i$ are all equal one.)

Using the proof of Proposition \ref{corr}, we deduce the $L$-degree of $\widehat{A}_K(M,L)$ for two-bridge knots.
\begin{corollary}
If $K=K(\al,\be)$ is a two-bridge knot, then the $L$-degree of $\widehat{A}_K(M,L)$ is given by $(\al - 1)/2$.
\end{corollary}

 In \cite{Nag}, this formula is given with $(\al-1)/2$
appearing as an upper bound for the $L$-degree of
$A_K(M,L)$, and we note that
one obtains a more precise statement by taking into account repeated factors, i.e. by replacing $A_K(M,L)$ with
$\widehat{A}_K(M,L)$.

This begs the question, given any small knot $K$,  to what extent do the two $A$-polynomials $A_K(M,L)$ and $\widehat{A}_K(M,L)$ coincide? Note that $A_K(M,L)$ is the polynomial that appears in the knot tables \cite{ChLi}, and for small knots, one can recover $A_K(M,L)$ from $\widehat{A}_K(M,L)$  by factoring and removing repeated factors.

 Using the calculations of Section \ref{double twist}, we determine the $M$-degree of $\widehat{A}_K(M,L)$ for  double twist knots.
These results can be used to identify knots for which the character variety of the boundary torus contains curves which are multiply covered under the map $X^*(M) \lto X(\partial M)$.
 In many cases, it can also be used to determine the degree of this covering on the various components. We illustrate this below in Section \ref{A-poly examples} and in the Appendix with specific examples, here we begin with the example of a torus knot.

Let $K$ be a $(p,q)$-torus knot different from the trefoil and set $M =  S^3 \sm \tau(K)$. We outline the argument that
 $A_K(M,L) \neq \widehat{A}_K(M,L)$, and a more detailed explanation can be found in \cite{BP}.
 Arguing as in Proposition 2.7 of \cite{CCGLS}, one can show that
$$A_K(M,L) = \begin{cases}
 1+LM^{2q} & \text{if $p=2$ and $q>2$ is odd,} \\
 (1+LM^{pq})(1-LM^{pq}) & \text{otherwise.} \\
 \end{cases}$$
In \cite{BP}, we enumerate the algebraic components of $X(M)$ and prove they all have multiplicity equal to one. As a consequence, it follows that
$$\widehat{A}_K(M,L) = \begin{cases}
 (1+LM^{pq})^{(p-1)(q-1)/4} (1-LM^{pq})^{(p-1)(q-1)/4} & \text{if $p$ and $q$ are odd,} \\
  (1+LM^{pq})^{p(q-1)/4} (1-LM^{pq})^{(p-2)(q-1)/4} & \text{if $p$ is even and $q$ is odd.} \\
\end{cases}$$
In particular, it follows that $A_K(M,L) \neq \widehat{A}_K(M,L)$ unless $K$ is the trefoil.

\begin{corollary} \label{cor-deg}
Consider the double twist knots $J(\ell, \pm m)$, where $\ell, m$ are positive integers.

\smallskip \noindent
{\rm Case I:} For $K=J(\ell,-m)$,  the $M$-degree of
$\widehat{A}_K(M,L)$ is given by
$$ \| {0}/{1} \|_{T} = \begin{cases}
  2\ell m &\text{if $\ell$ and $m$ are even,}\\
m(\ell m + 1) &\text{if $\ell$ is odd,}\\
 \ell (\ell m +1)  &\text{if $m$ is odd.}
  \end{cases}$$

\smallskip \noindent{\rm Case II:}
For $K=J(\ell,m)$ with $2 \le m \le \ell,$  the $M$-degree of
$\widehat{A}_K(M,L)$ is given by
$$ \| {0}/{1} \|_{T} = \begin{cases}
  2\ell m - 2 &\text{if $\ell$ and $m$ are even,}  \\
m^2 \, (\ell - 1)-(m-1)(m-2) &\text{if $\ell$ is odd,}  \\
 m (\ell-1)^2 - (\ell - m)+2(m - 1)  &\text{if $m$ is odd.}
  \end{cases}$$
\end{corollary}

\begin{proof}
This follows immediately by applying Theorems \ref{CS-formulaI} and \ref{CS-formulaII}  to  compute $\| 0/1 \|_{T}$ in both cases.
\end{proof}

\subsection{The knots $7_4$ and $8_{11}$} \label{A-poly examples}
In this subsection, we investigate the $\widehat{A}$-polynomials for the two knots $7_4$ and $8_{11}$. We first consider the $7_4$ knot $K,$ which can be realized as the two-bridge knot $K=K(15,11)$ and as the double twist knot $K = J(4,4)$. According to (independent) computations of J. Hoste and M. Culler listed in \cite{ChLi}, the mirror image of this knot has $A$-polynomial given by
\begin{equation*} 
\begin{split}
A_K (M,L) & \; = L^5 M^{22} - 3L^4 M^{22} + 3L^3 M^{22} - L^2 M^{22}
 + 7L^4 M^{20} - 10L^3 M^{20} \\
 & + 3L^2 M^{20}  + 4L^4 M^{18} + 3L^3 M^{18}
 - 3L^2 M^{18} - 6L^4M^{16}  + 21L^3M^{16} \\
 & - 2L^2 M^{16}  + L^4 M^{14}
 - 3L^3M^{14}  + 10L^2M^{14} + L M^{14}  + 3L^4 M^{12} \\
 & - 17L^3 M^{12}
+ 6L^2 M^{12}  - 2L M^{12}  - 2L^4 M^{10} + 6L^3 M^{10} - 17L^2 M^{10} \\
& + 3L M^{10}  + L^4 M^8 + 10L^3 M^8 - 3L^2 M^8  + L M^8 - 2L^3 M^6
+ 21L^2 M^6 \\
&- 6L M^6  - 3L^3 M^4   + 3L^2 M^4 + 4L M^4  + 3L^3 M^2
- 10L^2 M^2  + 7L M^2\\
& - L^3 + 3L^2  - 3L+1
\end{split}
\end{equation*}

Corollary \ref{cor-deg}
implies that $\widehat{A}_K(M,L)$ has $M$-degree 30 and $L$-degree 7, and using sage \cite{sage}, one can easily factor $A_K (M,L) $ as a product of two irreducible polynomials
\begin{equation*} 
\begin{split} A_K (M,L) & \; =  (L^2 M^8 - L M^8 + L M^6 + 2L M^4 + L M^2 - L + 1)
\cdot  (L^3 M^{14} - 2L^2 M^{14}   \\
&+ L M^{14}+ 6L^2 M^{12} - 2L M^{12} + 2L^2 M^{10} + 3L M^{10}
- 7L^2 M^8 + 2L M^8  \\
&+ 2L^2 M^6 - 7L M^6+ 3L^2 M^4 + 2L M^4 - 2L^2 M^2 + 6L M^2 + L^2 - 2L + 1).
\end{split}
\end{equation*}
From this we conclude that $\widehat{A}_K(M,L)$ has the same irreducible  factors as $A_K(M,L)$ but the first factor is repeated with multiplicity two.

Note that the correspondence between the Culler--Shalen seminorm and the $\widehat{A}$-poly-nomial applies to each component. This gives us new insight into the Culler--Shalen norm in this instance. For if we examine the $\widehat{A}$-polynomials of each factor, we see that the boundary slopes contributing to the first factor are 0 and 8 and the boundary slopes contributing to the second are 0 and 14. Thus we can decompose Ohtsuki's formula for the total norm into its component norms: the first with slopes 0 and $-8$, each with weight 2; and the second with slopes 0 and $-14$, with weights 2 and 1, respectively. In fact we are able to determine which incompressible surfaces contribute to each factor: the Seifert surface contributes to both factors, the two surfaces with slope $-8$ contribute to the first factor, and the surface with slope $-14$ contributes to the second factor.

We perform a similar analysis for the knot $8_{11}$, which is the two-bridge knot $K(27,10)$. Referring to the appendix below, we find there are six incompressible surfaces with slopes $0,0, -4, 6, 6,12$ and weights $1, 1, 3, 1, 4, 3,$ respectively.
The total Culler--Shalen norm is given by
$$ \| p/q \|_{T}= 2|p| + 3|p + 4q| + 5 |p  - 6q| + 3 |p - 12q|,$$
and $\widehat{A}_K(M,L)$ has  $M$-degree 78 and  $L$-degree  13.

According to the computations of J. Hoste and M. Culler listed in \cite{ChLi}, ${A}_K(M,L)$ has $M$-degree 66 and $L$-degree 11:
 \begin{align*}
A_K (M,L)  &=  M^{66}(L^5 - 2L^4 + L^3) +M^{64}( -11 L^5 + 21L^4 - 10L^3) + \cdots \\
& +M^{12}(- L^{11} + 11 L^{10} + 79 L^9 - 154 L^8 + 202 L^7 - 58 L^6 -185 L^5 + 3 L^4) + \cdots \\
& - M^2(10L^8 - 21L^7 + 11L^6) + L^8 - 2L^7 + L^6.
\end{align*}

Factoring $A_K(M,L)$ using sage \cite{sage}, we find that
$ A_K (M,L) = (M^6 + L) \cdot F(M,L),$
where
 \begin{align*}
F (M,L)  & =  M^{60}(L^5 - 2L^4 + L^3) + M^{58}(-11 L^5 + 21 L^4 - 10L^3)+ \cdots \\
&+ M^{12}(-L^{10} + 11 L^9  + 79 L^8 - 174 L^7 + 269 L^6  - 198 L^5  - 97 L^4 + L^3) +\cdots \\
&+ M^2(-10L^7 + 21L^6 - 11L^5) + L^7 - 2L^6 + L^5
\end{align*}
 is the irreducible polynomial  corresponding to the canonical component.

Note that $F(M,L)$ has $M$-degree 60 and $L$-degree 10, thus $$\widehat{A}_K(M,L) = (M^6 + L)^3 F(M,L).$$
It follows that the character variety consists of the canonical curve, one (or more) seminorm curves corresponding to the boundary slope 6, plus the component of abelians. This allows us to again split the total Culler--Shalen norm into its constituent norm and seminorm pieces:
\begin{align*}
\| p/q \|_{1} &= 2|p| + 3|p + 4q| + 2 |p  - 6q| + 3 |p - 12q|,\\
\| p/q \|_2 &= 3 |p - 6q|.
\end{align*}
Note that in this instance there are two incompressible surfaces with boundary slope 6, one of weight 4 and one of weight 1. It seems likely that the surface with slope 6 and weight 4 contributes with weight 2 to each factor, and that  the surface with slope 6 and weight 1 contributes to the seminorm factor only. However, we cannot make this conclusion using degree considerations alone. We are also unable to decide whether the multiplicity 3 factor $(L+M^6)^3$ of $\widehat{A}_K(M,L)$ comes from a single curve $X_j \subset X(M)$ in the character variety which is mapped with degree three into $X(\partial N_K)$, or instead if there are two or three curves in $X(M)$ with the same image under restriction $r\colon X(M) \to X(\partial M)$.

Comparing calculations of the Culler--Shalen seminorms with known results on the $A$-polynomials, we can sometimes determine the $\widehat{A}$-polynomial for other 2-bridge knots. While these techniques are reliably successful at identifying knots $K$ with
$A_K(M,L) \neq \widehat{A}_K(M,L)$, it is not always possible to determine $\widehat{A}_K(M,L)$. This approach works best
for knots whose character variety has only a few irreducible components, whereas the results of Ohtsuki, Riley, and Sakuma \cite{ORS} show that there exist two-bridge knots whose character varieties contain arbitrarily many components. In those cases, degree considerations alone
 are not sufficient for determining the multiplicities of $\widehat{A}_K(M,L)$.

 \subsection{Relation to the $\SLC$ knot invariant}
In this subsection, we examine the relationship between the knot invariant $\la'_\SLC(K)$ defined
in Equation \eqref{knotinv} and the $M$-degree of the
$\widehat{A}$-polynomial.
Proposition 6.6 of \cite{BZ2} implies that  for any one-dimensional component $X_i$ of $X(M)$ whose image $r(X_i)$ is also one-dimensional, $\deg_M A_{X_i}(M,L) =  \| 0/1\|_i$. Since $\widehat{A}_K(M,L)$ is the product of $\widehat{A}_{X_i}(M,L)$ over all such $X_i$,
it follows that  $\deg_M \widehat{A}_K(M,L)  =  \sum_{i} \| 0/1\|_i.$
On the other hand, for $K$ a small knot, Theorem \ref{surg-form} shows that there
are  positive integral weights $m_i$ such that
 $$\la'_\SLC(K) =  \tfrac{1}{2} \sum_{i} m_i \| 0/1\|_{i}.$$
Thus, for any small knot, we have $\la'_\SLC(K) >0$ if and only if $\deg_M \widehat{A}_K(M,L) >0$. The next result
shows that  $\la'_\SLC(K) >0$ for all small knots. This was proved for hyperbolic knots in \cite{C} and for torus knots in  \cite{BC1}.

\begin{theorem} \label{detectsmall}
Suppose $K$ is a nontrivial small knot in $S^3.$ Then its $\SLC$ Casson knot invariant satisfies
$\la'_\SLC (K) >0.$
\end{theorem}

\begin{proof}
Every knot is either a torus knot, a hyperbolic knot, or a satellite knot. As satellite knots are not small, the hypothesis implies that $K$ is either a torus knot or hyperbolic. Suppose first that $K$ is a $(p,q)$ torus knot. Since $\la'_\SLC(\overline{K})=\la'_\SLC({K})$, we can assume $p,q>0$ here. By Corollary 5.10 (i) \cite{BC1}, we see that $\la'_\SLC (K) = \tfrac{1}{4}pq(p-1)(q-1) >0.$

Now suppose that $K$ is a hyperbolic knot. The discrete faithful representation $\hat{\varrho} \colon \pi_1(M) \lto PSL(2,\CC)$ lifts to an $\SLC$ representation, and any component of $X(M)$ containing a lift is one-dimensional and is a norm curve
(see Definition \ref{CS-seminorm}).

By the surgery theorem, for each one-dimensional component $X_i$ of $X(M)$ such that $r(X_i)$ is one-dimensional, we have positive integral weights $m_i$, and the $\SLC$ Casson knot invariant
$\la'_\SLC(K)$ is equal to $\frac{1}{2}  \|0/1\|_T = \frac{1}{2} \sum_i m_i \| \lng \|_i.$ When $K$ is a small hyperbolic knot, any norm curve $X_i$ will have  $ \| 0/1\|_i >0,$ and since each $m_i \geq 1,$ this implies that $\la'_\SLC(K) >0$.
\end{proof}

It follows that $\deg_M \widehat{A}_K(M,L) >0$ for all nontrivial small knots.
Note that in the special case of a two-bridge knot $K$, Proposition \ref{weights} gives the more precise statement that
 $$\la'_\SLC(K) = \tfrac{1}{2} \deg_M \widehat{A}_K(M,L).$$
 In general, one can prove that $\deg_M A_K(M,L) >0$ for any nontrivial knot $K$ in $S^3$ (see \cite{B}),
but we do not know whether the same is true for the $\widehat{A}$-polynomial. This leads to the question: Is $\deg_M \widehat{A}_K(M,L) >0$ for every nontrivial knot $K$?

It is interesting to compare the $SU(2)$ and $\SLC$ Casson  knot invariants here. In contrast to the $\SLC$ knot invariant, the $SU(2)$ knot invariant $\la'_{SU(2)}(K)$ vanishes for many knots, including a number of  two-bridge knots.
 As the surgery formula  \ref{surg-form} only applies to small knots, it is not immediately evident how to define the  $\SLC$ knot invariant $\la'_{\SLC}(K)$ in general, and an interesting problem is  to extend $\la'_\SLC(K)$ to all knots and to determine whether it is nontrivial for all knots. This is quite likely related to the above question about nontriviality of $\deg_M \widehat{A}_K(M,L)$, and we hope to address these questions in future research.

\appendix

\section{Two-bridge knots up to 8-crossings}
In Table \ref{boundary} below, we collect all the information needed to determine the $\SLC$ Casson invariant for Dehn surgeries on knots with up to 8 crossings. Listed are the boundary slopes and Culler--Shalen weights. (Note that boundary slopes corresponding to more than one incompressible surface are listed multiply, with weights for each surface given.) Also listed are $\deg_M$ and $\deg_L$ of $A_K(M,L)$, and for knots with $A_K(M,L) \neq \widehat{A}_K(M,L)$, we also list  $\deg_M$ and $\deg_L$ of  $\widehat{A}_K(M,L)$. Knots with $A_K(M,L) \neq \widehat{A}_K(M,L)$ are indicated by $K^\star$.

\newpage

\begin{table}[h!] \footnotesize
\begin{tabular}{cc}
\begin{minipage}{0.5\textwidth}
\begin{tabular}{||r|ccc|rr||}
\hline &&&&& \\
 $K$ & &$\deg_M$ & $\deg_L$ & slope & weight \\ \hline
$3_1$ &$A_K$& 6 & 1 & 0 & 0\\
&& && 6 & 1\\ \hline

$4_1$ &$A_K$& 8 & 2 & 0 & 0\\
&& & & $-4$ & 1\\
&& & & 4 & 1\\ \hline

$5_1^\star$ &$A_K$& 10 & 1 & 0 & 0\\
&$\widehat{A}_K$& 20& 2& 10 & 2\\ \hline

$5_2$ &$A_K$& 14 & 3 & 0 & 1\\
&& & & 4 & 1\\
&& & & 10 & 1\\ \hline

$6_1$ &$A_K$& 16 & 4 & 0 & 1\\
&& & & $-4$ & 2\\
&& & & 8 & 1\\ \hline

$6_2$ &$A_K$& 30 & 5 & 0 & 0\\
&& & & 2 & 1\\
&& & & $-4$ & 1\\
&& & & 8 & 3\\ \hline

$6_3$ &$A_K$& 28 & 6 & 0 & 0\\
&& & & 2 & 1\\
&& & & $-2$ & 1\\
&& & & 6 & 2\\
&& & & -6 & 2\\ \hline

$7_1^\star$ &$A_K$& 14 & 1 & 0 & 0\\
&$\widehat{A}_K$& 42& 3 & 14 & 3\\ \hline

$7_2$ &$A_K$ & 22 & 5 & 0 & 2\\
&& & & 4 & 2\\
&& & & 14 & 1\\ \hline

$7_3$ &$A_K$& 52 & 6 & 0 & 1\\
&& & & $-8$ & 3\\
&& & & $-14$ & 2\\ \hline

$7_4^\star$ &$A_K$& 22 & 5 & 0 & 4\\
&$\widehat{A}_K$& 38& 7& $-8$ & 1\\
&& & & $-8$ & 1\\
&& & & $-14$ & 1 \\ \hline

  $7_5$ &$A_K$& 68 & 8 & 0 & 1\\
&& & & 4 & 1\\
&& & & 6 & 2\\
&& & & 10 & 1\\
&& & & 14 & 3\\ \hline

\end{tabular}
  \end{minipage}

\begin{minipage}{0.5\textwidth}
 \begin{tabular}{||r|rcc|rr||}

\hline &&&&& \\
$K$ && $\deg_M$ & $\deg_L$ & slope & weight \\ \hline

$7_6$ &$A_K$& 54 & 9 & 0 & 0\\
&& & & 0 & 1\\
&& & & 4 & 1\\
&& & & $-4$ & 2\\
&& & & 6 & 2\\
&& & & 10 & 3\\ \hline

$7_7^\star$ &$A_K$& 38 & 7 & 0 & 0\\
&$\widehat{A}_K$& 48& 10 & 0 & 1\\
&& & & 0 & 1\\
&& & & $-4$ & 1\\
&& & & $-4$ & 1\\
& && & 6 & 4\\
& & && $-8$ & 2\\ \hline

$8_1$ &$A_K$& 24 & 6 & 0 & 2\\
&& & & -4 & 3\\
&& & & 12 & 1\\ \hline

$8_2$ &$A_K$& 76 & 8 & 0 & 0\\
&& & & $-4$ & 1\\
&& & & 6 & 2\\
& && & 12 & 5\\ \hline

$8_3$ &$A_K$& 32 & 8 & 0 & 4\\
& && & 8 & 2\\
& && & -8 & 2\\ \hline

 $8_4$ &$A_K$& 58 & 9 & 0 & 1\\
& & && $-2$ & 1\\
&& & & 8 & 6\\
&& & & $-8$ & 1\\ \hline

$8_6$ &$A_K$& 62 & 11 & 0 & 1\\
& && & 2 & 3\\
& & && $-4$ & 2\\
&& & & 6 & 2\\
& && & 12 & 3\\ \hline

$8_7$ &$A_K$& 70 & 11 & 0 & 0\\
& & && $-2$ & 3\\
&& & & 6 & 2\\
&& & & $-6$ & 2\\
& && & $-10$ & 4\\ \hline

\end{tabular}
\vspace*{0.95cm}
\end{minipage}
\end{tabular}
\end{table}

 \newpage

\begin{table}[h!] \footnotesize

\begin{tabular}{cc}
\begin{minipage}{0.5\textwidth}
\begin{tabular}{||r|rcc|rr||}
\hline &&&&& \\
 $K$ && $\deg_M$ & $\deg_L$ & slope & weight \\ \hline

 $8_8$ &$A_K$& 60 & 12 & 0 & 0\\
& & && 2 & 3\\
&& & & $-4$ & 1\\
&& & & 6 & 4\\
& && & $-6$ & 1\\
& & && $-10$ & 2 \\  \hline

 $8_9$ &$A_K$& 60 & 12 & 0 & 0\\
&& & & 2 & 3\\
&& & & $-2$ & 3\\
& && & 4 & 4\\
& && & $-4$ & 4\\ \hline

$8_{11}^\star$ &$A_K$& 66 & 11 & 0 & 1\\
&$\widehat{A}_K$& 78& 13 & 0 & 1\\
& && & $-4$ & 3\\
& && & 6 & 1\\
& & && 6 & 4\\
& & && 12 & 3\\ \hline

\end{tabular}
\vspace*{2.4cm}
\end{minipage}

\begin{minipage}{0.5\textwidth}
 \begin{tabular}{||r|rcc|rr||}
 \hline &&&&& \\
 $K$ &&$\deg_M$ & $\deg_L$ & slope & weight \\ \hline

$8_{12}$ &$A_K$& 72 & 14 & 0 & 0\\
&& & & 0 & 2\\
&& & & 4 & 1\\
& && & 4 & 2\\
& & && $-4$ & 1\\
&& & & $-4$ & 2\\
&& & & 8 & 3\\
& && & $-8$ & 3\\ \hline

$8_{13}$ &$A_K$& 72 & 14 & 0 & 1\\
&& & & 0 & 1\\
&& & & $-2$ & 1\\
& && & $-4$ & 2\\
& && & 6 & 6\\
& & && -6 & 1\\
& & && $-10$ & 2\\ \hline

$8_{14}$ &$A_K$& 86 & 15 & 0 & 1\\
&& & & 0 & 2\\
&& & & 4 & 1\\
& && & 4 & 2\\
& & && $-4$ & 3\\
& && & 6 & 1\\
& & && 8 & 1\\
& & && 12 & 4 \\ \hline

 \end{tabular}
\end{minipage}
\end{tabular}
  \caption{Boundary slopes and weights for two-bridge knots with up to 8 crossings.} \label{boundary}
 \end{table}

  \noindent
{\it Acknowledgements.} Both authors would like to thank Steve Boyer for several helpful discussions. We also appreciate the careful reading and thoughtful suggestions of the referee. The first author would also like to thank the Max Planck Institute for Mathematics in Bonn for its support.



\begin{thebibliography}{999}




\bibitem{B}
H. U. Boden, {\it  Nontriviality of the $M$-degree of the $A$-polynomial,}
2012 preprint, to appear in Proc. Amer. Math. Soc.

\bibitem{BC1}
H. U. Boden and C. L. Curtis, {\it The $SL_2({\mathbb C})$ Casson
invariant for Seifert fibered homology spheres and surgeries on twist knots,}
J. Knot Theory Ramific. {\bf 15} (2006), 813--837.

\bibitem{BC2}
H. U. Boden and C. L. Curtis, {\it Splicing and the $SL_2({\mathbb C})$ Casson invariant,}
Proc. Amer. Math. Soc. {\bf 136} (2008), no. 7, 2615--2623.

\bibitem{BP}
H. U. Boden and K. L. Petersen, {\it Notes on the $A$-polynomial and $\widehat{A}$-polynomial,}
2012 in preparation.

\bibitem{BLP} M. Boileau, B. Leeb, and J. Porti,
{\it Geometrization of 3-dimensional orbifolds,} Ann. Of Math. {\bf 162} (2005) 195--290.

\bibitem{BZ0}
S. Boyer,
{\it Private communication.}


\bibitem{BZ1}
S. Boyer and X. Zhang,
{\it On the Culler--Shalen seminorms and Dehn filling,}
Ann. of Math. (2) {\bf 148} (1998), 737--801.

\bibitem{BZ2}
S. Boyer and X. Zhang,
{\it A proof of the finite filling conjecture,}
J. Diff. Geom. {\bf 59} (2001), 87--176.


\bibitem{BW}
M. Brittenham and Y.-Q. Wu, {\it The classification of exceptional Dehn surgeries on 2-bridge knots,}
Comm. Anal. Geom. {\bf 9} (2001), 97--113.


\bibitem{Bu}
G. Burde,
{\it  $SU(2)$ representation spaces for two-bridge knot groups,}
Math. Ann. {\bf 288} (1990), 103--119.

\bibitem{BuZi}
G. Burde and H. Zieschang,
{Knots,}
{\it de Gruyter Studies in Mathematics,} {\bf 5}, Walter de Gruyter \& Co., Berlin, 1985.


\bibitem{ChLi}
J. C. Cha and C. Livingston, {\it KnotInfo: Table of Knot Invariants,} online at {\tt http://www.indiana.edu/~knotinfo},  January 31, 2012.

\bibitem{CCGLS}
D. Cooper, M. Culler, H. Gillet, D. D. Long, and P. B. Shalen,
{\it  Plane curves associated to character varieties of 3-manifolds,}
Invent. Math. {\bf 118} (1994),   47--84.


\bibitem{CL1}
D. Cooper and D. D. Long,
{\it Remarks on the $A$-polynomial of a knot,}
J. Knot Theory Ramific. {\bf 5} (1996), 609--628.

\bibitem{CL2}
D. Cooper and D. D. Long,
{\it The $A$-polynomial has ones in the corners,}
Bull. Lond. Math. Soc. {\bf 29} (1997), 231--238.

\bibitem{CS1}
M. Culler and P. B. Shalen,
{\it Varieties of group representations and splittings of 3-manifolds,}
Annals of Math. {\bf 117} (1983), 109--146.

\bibitem{CS2}
M. Culler and P. B. Shalen,
{\it Bounded separating surfaces in knot manifolds,}
Invent. Math. {\bf 75} (1984), 537--545.

\bibitem{CGLS}
M. Culler, C. McA. Gordon, J. Luecke, and P. B. Shalen,
{\it  Dehn surgery on knots,}
Annals of Math. {\bf 125} (1987), 237--300.

\bibitem{C}
C. L. Curtis, {\it  An intersection theory count of the
$\SLC$-representations of the fundamental group of a 3-manifold,}
Topology {\bf 40} (2001), 773--787.

\bibitem{C1}
C. L. Curtis, {\it Erratum to ``An intersection theory count of the
$\SLC$-representations of the fundamental group of a 3-manifold,"}
Topology {\bf 42} (2003), 929.


\bibitem{HT}
A.  Hatcher and W. Thurston,
{\it  Incompressible surfaces in two-bridge knot complements,}
Invent. Math. {\bf 79} (1985), 225--246.

\bibitem{HS}
J. Hoste and P. Shanahan,
{\it  A formula for the $A$-polynomial of twist knots,}
J. Knot Theory Ramific. {\bf 13} (2004), 193--209.


\bibitem{KM}
T. Kim and T. Morifuji,
{\it Twisted Alexander polynomials and character varieties of 2-bridge knot groups,}
2010 preprint, arXiv: {\tt math.GT 1006.4285}.

\bibitem{Le}
T. Le,
{\it  Varieties of representations and their cohomology-jump subvarieties for knot groups,}
Russian Acad. Sci. Sb. Math. {\bf 78} (1994), 187--209.

\bibitem{MPL}
M. Macasieb, K. Petersen, and R. van Luick,
{\it  On character varieties of two-bridge knot groups,}
Proc. Lond. Math. Soc. (3) {\bf 103} (2011), 473--507.

\bibitem{M}
T. Mattman,
{\it  The Culler--Shalen seminorms of pretzel knots,}
Ph.D. Thesis, McGill
University, Montreal (2000), {\tt http://www.csuchico.edu/math/mattman}


\bibitem{Nag}
F. Nagasato,
{\it  On a behavior of a slice of the $SL_2(\CC)$-character variety of a knot group under the connected sum,}
 Topology Appl. {\bf 157} (2010), 182--187.

\bibitem{O}
T. Ohtsuki, {\it Ideal points and incompressible surfaces in
two-bridge knot complements,} J. Math. Soc. Japan {\bf 46} (1994),
51--87.

\bibitem{ORS}
T. Ohtsuki, R. Riley, and M. Sakuma, {\it Epimorphisms between 2-bridge link groups,}
The Zieschang Gedenkschrift, 417--450, Geom. Topol. Mono., {\bf 14} 2008.

\bibitem{Riley}
R. Riley, {\it Nonabelian representations of 2-bridge knot groups,}
Quart. J. Math., {\bf 35} (1984) 191--208.

\bibitem{R}
D. Rolfsen,
Knots and Links,
{\it Mathematics Lecture Series} {\bf 7},
Publish or Perish, Berkeley, CA (1976).

\bibitem{S}
N. Saveliev, Invariants for homology 3-spheres,
{\it Encylopaedia of Math. Sci.} {\bf 140}, Springer-Verlag, Berlin Heidelberg (2002).

\bibitem{Sh}
P. Shalen,  Representations of 3-manifold groups,
{\it Handbook of Geometric Topology,} 955--1044
North Holland, Amsterdam (2002).


\bibitem{sage} W. A. Stein et al., {\it Sage Mathematics Software (Version 4.8),}
   The Sage Development Team, 2012, {\tt http://www.sagemath.org.}

\end{thebibliography}
\end{document}